\documentclass[12pt]{amsart}
\usepackage[all]{xy}
\usepackage{graphicx}
\usepackage{tikz}  
 
\usepackage{amsmath} 
\usepackage{amssymb}
\usepackage{amsfonts}
\usepackage{mathrsfs}
\usepackage{mathtools}
\usepackage{mathabx}

\usepackage{amsmath,amssymb,amsfonts,amsthm,stackengine,graphicx,caption}	
\usepackage{color}	

\def \N {\mathbb N}

\def \R {\mathbb R}
\def \C {\mathbb C}
\def \D {\mathbb D}
\def \T {\mathbb T}
\def \P {\mathbb P}
\def \ran {\mathrm{Ran}}
\def \eps {\varepsilon}

\renewcommand{\Re}{\operatorname{Re}}

\newcommand\abs[1]{\left|#1\right|}
\newcommand\dx{\mathrm{d}}

\newcommand{\norm}[1]{\Vert#1\Vert}
\newcommand{\bignorm}[1]{\bigl\Vert#1\bigr\Vert}
\newcommand{\Bignorm}[1]{\Bigl\Vert#1\Bigr\Vert}
\newcommand{\biggnorm}[1]{\biggl\Vert#1\biggl\Vert}
\newcommand{\Biggnorm}[1]{\Biggl\Vert#1\Biggr\Vert}

\newcommand{\rittE}{\textrm{Ritt}_E}

\newcommand{\Rad}[1]{{\rm Rad}(#1)}
\newcommand{\Gauss}{{\rm Gauss}}
\newcommand{\rara}[1]{{\rm Rad}({\rm Rad}(#1))}

\newcommand{\HI}{H^\infty}

\theoremstyle{plain}
\newtheorem{theorem}{Theorem}[section]
\newtheorem{proposition}[theorem]{Proposition}
\newtheorem{corollary}[theorem]{Corollary}
\newtheorem{definition}[theorem]{Definition}
\newtheorem{lemma}[theorem]{Lemma}
\theoremstyle{definition}
\newtheorem{remark}[theorem]{Remark}

\title{SQUARE FUNCTIONS ASSOCIATED WITH RITT$_E$ OPERATORS}
\author{Oualid Bouabdillah}
\email{oualid.bouabdillah@univ-fcomte.fr,}
\address{Laboratoire de Math\'ematiques de Besan\c con, UMR 6623, 
CNRS, Universit\'e Bourgogne Franche-Comt\'e,
25030 Besan\c{c}on Cedex, FRANCE}
\date{}
\begin{document}
\iffalse
%------------------------ ABSTRACT AVEC MACRO ----------------------
\begin{abstract}
    For a subset $E = \{\xi_1, ..., \xi_N\}$ of the unit circle $\mathbb{T}$, the notion of Ritt$_E$ operators on a Banach space and their functional calculus on generalized Stolz domains was developed and studied in \cite{BLM}. 

    In this paper, we define a quadratic functional calculus for a Ritt$_E$ operator on $E_r$, by a decomposition of type Franks-McIntosh. We show that with some hypothesis on the cotype of $X$, this notion is equivalent to the existence of a bounded functional calculus on $E_r$.

    We define for a Ritt$_E$ operator on a Banach space $X$ and for any positive real number $\alpha$ and for any $x \in X$
    $$
        \norm{x}_{T,\alpha} = \lim\limits_{n\rightarrow \infty}\Bignorm{\sum\limits_{k=1}^n k^{\alpha - 1/2} \varepsilon_k \otimes T^{k-1}\prod\limits_{j=1}^N(I-\overline{\xi_j}T)^\alpha(x)}_{\Rad(X)}
    $$
    We show that, under the condition of finite cotype of $X$, a Ritt$_E$ operator admits a quadratic functional calculus if and only if the estimates $\norm{x}_{T,\alpha} \lesssim \norm{x}$ hold for both $T$ and $T^*$.
    We finally prove the equivalence between these square functions.

\end{abstract}
\fi

\maketitle

%------------------------ ABSTRACT SANS MACRO ----------------------
\begin{abstract}
    For a subset $E = \{\xi_1, ..., \xi_N\}$ of the unit circle $\mathbb{T}$, the notion of Ritt$_E$ operators on a Banach space and their functional calculus on generalized Stolz domains was developed and studied in \cite{BLM}. %arXiv:2203.05373

    In this paper, we define a quadratic functional calculus for a Ritt$_E$ operator on $E_r$, by a decomposition of type Franks-McIntosh. We show that with some hypothesis on the cotype of $X$, this notion is equivalent to the existence of a bounded functional calculus on $E_r$.

    We define for a Ritt$_E$ operator on a Banach space $X$ and for any positive real number $\alpha$ and for any $x \in X$
    $$
        \Vert{x}\Vert_{T,\alpha} = \lim\limits_{n\rightarrow \infty}\Bigl\Vert{\sum\limits_{k=1}^n k^{\alpha - 1/2} \varepsilon_k \otimes T^{k-1}\prod\limits_{j=1}^N(I-\overline{\xi_j}T)^\alpha(x)}\Bigr\Vert_{{\rm Rad}(X)}
    $$
    We show that, under the condition of finite cotype of $X$, a Ritt$_E$ operator admits a quadratic functional calculus if and only if the estimates $\Vert{x}\Vert_{T,\alpha} \lesssim \Vert{x}\Vert$ hold for both $T$ and $T^*$.
    We finally prove the equivalence between these square functions.

\end{abstract}
\vskip 0.8cm
\noindent
{\it 2000 Mathematics Subject Classification:} Primary 47A60, secondary	47B12, 47B01.

\smallskip
\noindent
{\it Key words:} Functional calculus, Sectorial operators, Ritt operators, $R$-boundedness, Square functions estimates.

\section{Introduction}

    In a previous article \cite{BLM}, the class of Ritt operators was extended to Ritt$_E$ operators. This is the class of operators $T$ 
    on some Banach space for which the spectrum $\sigma(T)$
    is included in the closed unit disc, the intersection of
    $\sigma(T)$ with the unit circle is included in a finite 
    set $E = \{\xi_j\}_{1\leq j \leq N} \subset \T$, with $N \geq 1$, and the set
    $$
        \Bigl\{\prod\limits_{j=1}^N (1-\overline{\xi_j}z) R(z,T) : 1 <\abs{z}<C\Bigr\}
    $$
    is bounded, for any $C > 1$.

    On Banach spaces, Ritt operators are used 
    in the context of maximal regularity 
    for Cauchy problems (see \cite{Bl,KP}), 
    as well as for resolution of linear equations by iterative methods.
    They are closely related 
    to sectorial operators, which appear 
    naturally in semi-group theory (see \cite{Gol,Paz}).
    Several articles deal with their functional
    calculus (see for instance \cite{AFL,ALM,Bl,GT,KP,LM1,LMX,Ly,NZ,N,V}), either on Hilbert spaces, on $L^p$-spaces 
    or on more general Banach spaces.

    We recall that the spectrum of a Ritt operator $T$ intersects the unit circle $\T$ at most  at $\{1\}$. Further such an operator has its spectrum included in a Stolz domain, which can be defined as the convex hull of $\{1\} \cup r \D$, for a certain $r \in (0,1)$.
    Ritt$_E$ operators satisfy a similar property, with $\{1\}$ replaced by $E$.
    More generally, several important properties of Ritt operators
    extend to Ritt$_E$ operators, see \cite{BLM} for examples. In particular, we showed in the latter paper the existence of a bounded $\HI$-functional calculus for polynomially bounded Ritt$_E$
    operators satisfying an appropriate $\R$-boundedness property.

    \iffalse
    operators, with respect to more general Stolz domains. The proofs mostly mimicked the Ritt case, by considering contours of integration respectively to the geometry of the frontier of these domains.
\fi

    %We gave a new proof of Delaubenfels theorem (see \cite[Theorem 4.4]{dL}), which treats about these operators and generalized it to some Banach spaces, in particular $L^p$ spaces (see \cite[Theorem 3.9]{BLM} and \cite[Theorem 4.7]{BLM}). 

    For a Ritt operator $T$, there is a close link between 
    $\HI$ functional calculus of $T$ and square functions 
    estimates associated with $T$. Under certain 
    conditions, the existence of a bounded $\HI$ calculus for $T$ 
    is equivalent to square functions estimates for $T$ and $T^*$, 
    see \cite{LM1}. 
    %Those are used to give solutions to some elliptic PDEs, in Harmonic analysis and in problems of maximal regularity as well. 
    Note that the definition of square functions associated with 
    a Ritt operator need Rademacher means (see e.g \cite{HVVW}).

    The purpose of this article is to introduce square functions adapted to Ritt$_E$ operators, and to show characterizations
    of the existence of a bounded $\HI$ bounded functional calculus
    with square function estimates in this context. 
    This requires the introduction of a stronger notion of bounded 
    functional calculus of an operator (named quadratic functional calculus) which, under the hypothesis of 
    finite cotype, coincides with the usual one. 
    This result relies on an explicit Franks-McIntosh decomposition adapted to Ritt$_E$ operators, which is detailed in section 4.
    
    %This relies on a Franks-McIntosh decomposition adapted to Ritt$_E$ operators.
    
    Let $E \subset \T$ be a finite family and $T$ be a Ritt$_E$ operator (see \cite[Definition 2.1]{BLM}) on $X$. Given $\alpha>0$, we define the square functions associated with $T$ as Rademacher averages of variables $(\eps_k)_{k\geq 1}$ by
    $$
        \norm{x}_{T,\alpha} = \lim\limits_{n\rightarrow \infty} \,\Bignorm{\sum\limits_{k = 1}^n k^{\alpha - 1/2} \eps_k \otimes T^{k-1} \prod\limits_{j=1}^N (1 - \overline{\xi_j}T)^\alpha x}.
    $$    
    In section 6, we show that the existence of a $\HI$ functional calculus of a Ritt$_E$ operator implies general square functions estimates and conversely, using the notion of quadratic functional calculus described in section 5. In the case where $X$ has property $(\Delta)$, we prove a stronger result.

    Finally, in section 7, we show the equivalence of all the square functions $\norm{.}_{T,\alpha}$ and $\norm{.}_{T,\beta}$, 
    for any $\alpha, \beta >0$, under hypothesis of finite cotype and reflexivity. This implies that under these conditions, more general square function estimates are equivalent to the existence of a $\HI$ bounded functional calculus.

    All the results of these two sections use approximation principles presented in section 3.

\newpage
%VERSION THESE, A JOUR DU 17 08 2024

\section{Preliminaries}
Throughout we let $X$ be a Banach space and we let $B(X)$ denote the Banach algebra of all bounded operators on $X$.
We define square functions 
associated with a Ritt$_E$ operator. These definitions are natural
extensions of square functions associated with Ritt operators,
for which we refer to \cite{ALM, LM1}. 
%We use Rademacher averages from Definition \ref{defRad}.
Let $(\eps_k)_{k\geq 1}$ be a sequence of independant Rademacher random variables on some probability space $(\Omega,\P)$ and let $\Rad{X}$ be the subspace of the Bochner space $L^2(\Omega,X)$ spanned by the set $\{\eps_k\otimes x : k \geq 1, x \in X\}$. For any finite families $x_1, ... , x_n$ of elements of $X$, we have
\begin{equation}
    \Bignorm{\sum\limits_{k=1}^n \eps_k\otimes x_k}_{\Rad{X}} = \left(\int_{\Omega}\bignorm{\sum\limits_{k=1}^n\eps_k(u)x_k}_X^2\dx\P(u)\right)^{1/2}
\end{equation}

This norm provides a way to deal with $\mathcal{R}$-boundedness. We give the following definition of this notion, and define $\mathcal{R}$-Ritt$_E$ operators.

\begin{definition}
    Let $X$ be a Banach space and consider a subset $F \subset B(X)$ of operators.
    We say that $F$ is $\mathcal{R}$-bounded (respectively $\gamma$-bounded) if there exists a constant $C>0$ such that for all finite sequences $(T_k)_{k\geq 1}$ in $F$ and $(x_k)_{k\geq 1}$ in $X$, 
        \begin{equation}\label{RBDDNESS}
            \Bignorm{\sum\limits_{k=1}^n \varepsilon_k \otimes T_k x_k}_{\Rad{X}} \leq C \Bignorm{\sum\limits_{k=1}^n \eps_k \otimes x_k}_{\Rad{X}}
        \end{equation}
\end{definition}

The following fact 

\begin{proposition}\label{R-RITTE}
    An operator $T : X \to X$ on a Banach space $X$ is $\mathcal{R}$-Ritt$_E$ if and only if there exists 
    $r\in (0,1)$ such that
        $$
            \sigma(T) \subset \overline{E_r}
        $$
    and for all $s\in(r,1)$, the set
        $$
            \biggl\{R(z,T) \prod\limits_{j=1}^N(\xi_j-z)\, :\, z\in D(0,2)\setminus \overline{E_s}\biggr\}
        $$
    is $\mathcal{R}$-bounded.
\end{proposition}

We may now introduce square functions associated with Ritt$_E$ operators, and explicit them in particular cases (such as in Hilbert spaces of Banach lattices of finite cotype).

\begin{definition}\label{SF}
Let $T : X \rightarrow X$ be a Ritt$_E$ operator and let $\alpha > 0$.
%Let $(\varepsilon_k)_{k\geq 1}$  be a sequence of independent  Rademacher variables. 
The square function $\norm{\,\cdotp}_{T,\alpha}$ associated with $T$ is defined, for any $x \in X$, by
$$
\norm{x}_{T,\alpha} = \lim\limits_{n\rightarrow \infty}\Bignorm{\sum\limits_{k=1}^n k^{\alpha - 1/2} \varepsilon_k \otimes T^{k-1}\prod\limits_{j=1}^N(I-\overline{\xi_j}T)^\alpha(x)}_{\Rad{X}}.
$$
\end{definition}

\begin{remark}

(1) If $T$ is a Ritt$_E$ operator, then $I-\overline{\xi_j}T$ is a sectorial operator for each $j=1,\ldots,N$, by \cite[Lemma 2.4]{BLM}. Thus, $(I-\overline{\xi_j}T)^\alpha$ is a well-defined bounded operator on $X$, using the classical definition of fractional powers of sectorial operators, which can be found e.g in \cite[Chapter 3]{Haase}.

\smallskip
(2) The square functions are well-defined, because the sequence of 
norms
$$
\Bignorm{\sum\limits_{k=1}^n k^{\alpha - 1/2} \varepsilon_k \otimes 
T^{k-1}\prod\limits_{j=1}^N(I-\overline{\xi_j}T)^\alpha(x)}_{\Rad{X}},\quad n\geq 1,
$$
is nondecreasing. However, $\norm{x}_{T,\alpha}$ may be infinite.

\smallskip
(3) 
If $X = H$ is a Hilbert space, the square functions are equal to 
$$
\norm{x}_{T,\alpha} = \left(\sum\limits_{k=1}^\infty k^{2\alpha-1}
\Bignorm{T^{k-1}\prod\limits_{j=1}^N(I-\overline{\xi_j}T)^\alpha (x)}^2\right)^{1/2}.
$$
%This ``explains" why the functionals $\norm{\,\cdotp}_{T,\alpha}$
%are called square functions.

\smallskip
(4) 
If $X$ is a Banach lattice with finite cotype, then we have
$$
\norm{x}_{T,\alpha} \approx \biggnorm{
\Bigl(\sum\limits_{k=1}^\infty k^{2\alpha-1}
\Bigl\vert T^{k-1}\prod\limits_{j=1}^N
(I-\overline{\xi_j}T)^\alpha (x)\Bigr\vert^2\Bigr)^{1/2}}_X.
$$
\end{remark}

The following remarkable result of Kaiser-Weis (see \cite[Corollary 3.4]{KCWL}) will be used further to link quadratic and bounded functional calculus.

\begin{lemma}\label{Kaiser}
Let $X$ be a Banach space with finite cotype. Then there exists a constant $C>0$ such that
$$
\Bignorm{\sum\limits_{k,l\geq1}\alpha_{kl}\eps_{k}\otimes\eps_l\otimes x_k}_{\rara{X}} \leq C \sup\limits_{k}\left(\sum\limits_{l} \abs{\alpha_{kl}}^2\right)^{1/2} \Bignorm{\sum\limits_{k} \eps_k \otimes x_k}_{\Rad{X}}
$$
for any finite family $(x_k)_k$ of $X$ and any finite
family $(\alpha_{kl})_{k,l \geq 1}$ of complex numbers.
\end{lemma}

\iffalse
    These functions give a proof of the converse of the previous implication, more precisely : 
    \begin{theorem}
        Let $T : X \rightarrow X$ a $\mathcal{R}$-Ritt$_E$ operator of $\mathcal{R}$-type $r \in (0,1)$. If $T$ and $T^*$ both satisfy uniform estimates : 
        $$
            \norm{x}_{T} \lesssim \norm{x} \hspace{1cm} \textrm{and} \hspace{1cm} \norm{y}_{T^*} \lesssim \norm{y}
        $$
        for $x \in X$ and $y \in X^*$, then $T$ admits a bounded $\HI(E_s)$ calculus for any $s \in (r,1)$.
    \end{theorem}
\fi

%When $X$ is a $K$-convex Banach space, then we have the following, which will be used in the end of this chapter.

We end these preliminaries with the stability of $\mathcal{R}$-boundedness by passing to the adjoint, when $X$ is $K$-convex.

\begin{theorem}\label{KcvRbdd}
Let $X$ be a $K$-convex Banach space and 
let $F \subset B(X)$ 
be an $\mathcal{R}$-bounded set of operators on $X$.     
Then the set $F^* = \{T^*, T \in F\}$ is 
$\mathcal{R}$-bounded as well.
\end{theorem}

\begin{proof}
            Let $(y_k)_k$ be a finite family of $X^*$. %We may apply Proposition \ref{DualRad}. 
            %-------------------AJOUT-----------------------
            By $K$-convexity, the spaces $\Rad{X^*}$ and $\Rad{X}^*$ are isomorphic.
            %-----------------------------------------------
            Thus we have the following estimates.
            \begin{align*}
                \Bignorm{\sum\limits_k \eps_k \otimes T_k^*(y_k)}_{\Rad{X}^*} &\lesssim \sup\{\abs{\sum\limits_k \langle{T_k^*(y_k),x_k\rangle}} : (x_k)_k \subset X, \norm{\sum\limits_{k}\eps_k \otimes x_k}_{\Rad{X}} \leq 1\}\\
                &= \sup\{\abs{\sum\limits_k \langle{y_k,T^*(x_k)\rangle}} : (x_k)_k \subset X, \norm{\sum\limits_{k}\eps_k \otimes x_k }_{\Rad{X}}\leq 1\}\\
                &\leq \Bignorm{\sum\limits_k \eps_k \otimes y_k}_{\Rad{X^*}} \times \Bignorm{\sum\limits_k \eps_k \otimes T_k(x_k)}_{\Rad{X}}\\
                &\leq \mathcal{R}(F) \Bignorm{\sum\limits_k \eps_k \otimes y_k}_{\Rad{X^*}}.
            \end{align*}
        \end{proof}

\newpage
\section{$\mathcal{R}$-Ritt$_E$ operators and approximation properties}

Many properties that we study for Ritt$_E$ operators $T : X \rightarrow X$ need to be proved first for the operators $\rho T$, with $\rho\in(0,1)$. These operators are easier to handle with, since their
spectrum is included in $\mathbb D$. However to pass from results on $\rho T$ to results on $T$
requires uniform estimates and approximation properties. This section is devoted to such issues.

\begin{lemma}\label{indRittRho}
Let $T$ be a Ritt$_E$ operator (respectively an $\mathcal{R}$-Ritt$_E$ operator). Let $r\in(0,1)$ such that 
\begin{itemize}
\item $\displaystyle{\sigma(T) \subset \overline{E_r}}$;
\item For all $s\in(r,1)$, there exists a constant $c>0$
such that 
$$
\|R(z,T)\| \leq\, \frac{c}{\prod\limits_{j=1}^N \abs{\xi_j - z}},
\qquad
z\in D(0,2)\setminus \overline{E_s}.
$$
\end{itemize}
Then for any $s\in(r,1)$, the set
\begin{equation}\label{Set}
\Bigl\{R(z,\rho T){\prod\limits_{j=1}^N({1-\overline{\xi_j}z}})\, 
:\,\rho \in (0,1],\ z \notin\overline{E_s},
\ \abs{z} \leq 2\Bigr\}
\end{equation}
is bounded (respectively $\mathcal{R}$-bounded).
\end{lemma}

\begin{proof}
We prove the $\mathcal{R}$-Ritt$_E$ case, the proof of the Ritt$_E$ case being similar. Let $a = \frac s 2$. 
For any $\rho \in [a,1]$ and any $z\notin\overline{E_s}$ we have 
$\frac{z}{\rho}\notin\overline{E_s}$ and 
$R(z,\rho T) = 
\frac 1 \rho R(\frac z \rho,T)$. Then we write 
$$
R(z,\rho T) \prod\limits_{j=1}^N (1- \overline{\xi_j}z) = 
\frac 1 \rho \left(\prod\limits_{j=1}^N 
\cfrac{1- \overline{\xi_j}z}{1-\overline{\xi_j}
\cfrac z \rho}\right) \left(R\bigl(\tfrac z \rho , T\bigr)\prod\limits_{j=1}^N 
\bigl(1-\overline{\xi_j} \tfrac z \rho\bigr)\right).
$$
If $\abs{z}\leq 2$ and $\rho \in [a,1]$, 
then we have $\bigl\vert\frac{z}{\rho}\bigr\vert\leq \frac{2}{a}$. 
Hence by %Remark \ref{R}, 
\cite[Remark 2.3]{BLM}, 
%----------------(OU \cite[Theorem 8.5.2]{HVVW} ???)-----------
the set
$$
\Bigl\{R\bigl(\tfrac z \rho , T\bigr)\prod\limits_{j=1}^N \bigl(1-\overline{\xi_j}\,\tfrac z \rho\bigr)\,
:\,\rho \in [a,1],\ z \notin\overline{E_s},
\ \abs{z} \leq 2\Bigr\}
$$
is $\mathcal{R}$-bounded.

Further there exists a constant $K\geq 1$ such that 
$$
\vert \xi_j -z\vert\leq K\bigl\vert \xi_j-\tfrac{z}{\rho}\bigr\vert
$$
for all $z\notin\overline{E_s}$, $j\in\{1,\ldots,N\}$ and 
$\rho \in [a,1]$. Consequently, 
$$
\biggl\{\cfrac 1 \rho \prod\limits_{j=1}^N \cfrac{1- \overline{\xi_j}z}{1-\overline{\xi_j}\cfrac z \rho} 
\,:\,\rho \in [a,1],\ z \notin\overline{E_s},
\ \abs{z} \leq 2\biggr\}
$$
is a bounded subset of $\mathbb C$. 
We deduce the $\mathcal{R}$-boundedness of the set 
$$
\Bigl\{R(z,\rho T){\prod\limits_{j=1}^N({1-\overline{\xi_j}z}})\, 
:\,\rho \in [a,1],\ z \notin\overline{E_s},
\ \abs{z} \leq 2\Bigr\}
$$

We study now the case $\rho \in (0,a)$. 
Let $\mathcal K = \overline{D} (0,\frac 1 2)$. The mapping 
$\lambda \in \D \mapsto (1-\lambda T)^{-1}$ is well 
defined and holomorphic on $\mathcal K$, which is a compact of $\D$. Then, by 
\cite[Proposition 2.8]{W}, the set
$$
F := \bigl\{(1-\lambda T)^{-1} \,: \, \lambda \in \mathcal K\bigr\}
$$
is $\mathcal{R}$-bounded.
But for any $z \notin \overline{E_s}$ and $\rho \in (0,a]$, 
we have $\abs{\cfrac {\rho} z} \leq \cfrac 1 2$. 
Hence the set
$$
\bigl\{z R(z,\rho T) \, : \, z \notin{\overline{E_s}},\ \rho \in (0,a]\bigr\}
$$
is included in $F$, hence is $\mathcal{R}$-bounded. Then 
$$
\{R(z,\rho T) \, : \, z \notin{\overline{E_s}},\ \rho \in (0,a]\}
$$
is $\mathcal{R}$-bounded. Finally, the set
$$
\Bigl\{R(z,\rho T){\prod\limits_{j=1}^N({1-\overline{\xi_j}z}})\, 
:\,\rho \in (0,a],\ z \notin\overline{E_s},
\ \abs{z} \leq 2\Bigr\}
$$
is $\mathcal{R}$-bounded. We therefore obtain that the set (\ref{Set})
is $\mathcal{R}$-bounded. 
\end{proof}

\begin{lemma}\label{phirhoT}
Let $T : X \rightarrow X$ be a Ritt$_E$ operator of type $r\in (0,1)$, let 
$s\in(r,1)$ and let $\phi\in H^\infty_0(E_s)$. Then 
$$
\phi(\rho T)\underset{\rho \rightarrow  1}{\longrightarrow} \phi(T)
$$
in the norm topology of $B(X)$.
\end{lemma}

\begin{proof}
Let $u\in (r,s)$. By Lemma \ref{indRittRho},
all the operators $\rho T$ are Ritt$_E$ of type $r$, and we have
$$
\phi(\rho T) = \cfrac{1} {2\pi i} \int_{\partial E_{u}} \phi(\rho z) R(z,\rho T) \dx z,\qquad \rho\in(0,1].
$$
Since $\phi \in \HI_0(E_s)$, 
there exist positive real numbers $C_1>0$, 
and $s_1,...,s_n >0$ such that 
\begin{equation}\label{C1}
\abs{\phi(z)} \leq C_1 
\prod\limits_{j=1}^N \abs{\xi_j-z}^{s_j}
\end{equation}
for all $z \in E_s$.
Hence according to Lemma \ref{indRittRho}, 
there exists $C_2>0$ such that 
$$
\norm{\phi(\rho z) R(z,\rho T)} \leq C_2 
\prod\limits_{j=1}^N \abs{\xi_j-z}^{s_j-1}
$$
for all $z\in E_s$ and $\rho\in(0,1)$.
Since the right handside function is integrable 
on $\partial E_u$, it follows from 
Lebesgue's dominated convergence theorem that
$$
\phi(\rho T) = \cfrac{1} {2\pi i} 
\int_{\partial E_{u}} \phi(\rho z) R(z,\rho T) 
\dx z \underset{\rho \rightarrow 1}{\longrightarrow} 
\cfrac{1} {2\pi i} \int_{\partial E_{u}} \phi(z) R(z,T) 
\dx z = \phi(T).
$$
\end{proof}

It is shown in \cite[Theorem 2.9]{BLM} that if $T$ is a Ritt$_E$ operator, then the set 
$$
\Bigl\{nT^{n-1}\prod\limits_{j=1}^N
(\xi_j-T)
\,:\, n\geq 1\Bigr\}
$$
is bounded.
The next statement is an extension of the latter result.

\begin{proposition}\label{Rbddalpha}
Let $T : X \rightarrow X$ be a Ritt$_E$ operator 
(resp. an $\mathcal{R}$-Ritt$_E$ operator). For any $\alpha >0$, the set
$$
\Bigl\{n^\alpha(\rho T)^{n-1}
\prod\limits_{j=1}^N(I-\rho \overline{\xi_j}T)^\alpha 
\,:\, n \geq 1, \rho \in (0,1]\Bigr\}
$$
is bounded (resp. $\mathcal{R}$-bounded).
\end{proposition}

\begin{proof} %(Proposition \ref{Rbddalpha})
We prove the $\mathcal{R}$-Ritt$_E$ case only, the case Ritt$_E$ being similar and simpler. 
So assume that $T$ is $\mathcal{R}$-Ritt$_E$. Let $r\in(0,1)$ such that the conditions of Lemma \ref{indRittRho} hold true - in particular, 
$r$ is $E$-large enough 
%---------DOIT ON REDONNER LA DEF ??? ----------------
(see \cite{BLM} for the definition). Fix some $s\in(r,1)$ 
and decompose the boundary of $\partial E_s$ as follows, 
with $\theta = \arccos(s)$. For any $k \in \{1,..,N\}$, let
$$
\Gamma_{k,1} := [\xi_k,se^{i\theta}\xi_k] \quad \mathrm{and} \quad
\Gamma_{k,2} := [se^{-i\theta}\xi_k,\xi_k].
$$
We define $\Gamma_{k}$ as the arc of circle of centre $0$ linking 
$se^{i\theta}\xi_k$ to $se^{-i\theta}\xi_{k+1}$, 
oriented counterclockwise, with the convention 
that $\xi_{N+1} = \xi_1$. Thus,
$\partial E_s$ is the concatenation of the $\Gamma_k$,
$\Gamma_{k,1}$ and $\Gamma_{k,2}$, for $k \in \{1,..,N\}$.

Consider a fixed number $\alpha >0$. For any integer 
$n\geq 1$ and any $\rho \in (0,1)$, the Dunford-Riesz 
functional calculus applied 
to the operator $(\rho T)^{n-1}\prod\limits_{j=1}^N(I-\overline{\xi_j} 
\rho T)^\alpha$ gives : 
\begin{equation}\label{FormulaContour}
(\rho T)^{n-1}\prod\limits_{j=1}^N(I-\overline{\xi_j} \rho T)^\alpha = \frac 1 {2\pi i} \int_{\partial E_s} \lambda^{n-1}\prod\limits_{j=1}^N(1-\overline{\xi_j} \lambda)^\alpha R(\lambda,\rho T) \dx \lambda.
\end{equation}
We multiply both sides by $n^\alpha$ and isolate $\prod\limits_{j=1}^N(1-\overline{\xi_j}\lambda) R(\lambda,\rho T)$ in the integral, which gives 
$$
n^{\alpha}(\rho T)^{n-1}\prod\limits_{j=1}^N(I-\overline{\xi_j} \rho T)^\alpha = \frac {n^\alpha} {2\pi i} \int_{\partial E_s} \lambda^{n-1}\prod\limits_{j=1}^N(1-\overline{\xi_j} \lambda)^{\alpha-1}  \prod\limits_{j=1}^N(1-\overline{\xi_j} \lambda)R(\lambda,\rho T) 
\dx \lambda.
$$

According to Lemma \ref{indRittRho}, the set
$$
F := \Bigl\{R(z,\rho T) \prod\limits_{j=1}^N (1-\overline{\xi_j}z) \,: \, 
\rho \in (0, 1) , z \in \partial E_{s} \setminus E\Bigr\}
$$
is $\mathcal{R}$-bounded. Hence if we 
show that the integrals 
$$
{\int_{\partial E_s}} n^\alpha \Bigl\vert
\lambda^{n-1} \prod\limits_{j=1}^N (1-\overline{\xi_j}\lambda)^{\alpha-1} 
\Bigr\vert\,\vert\dx \lambda\vert
$$
are uniformly bounded for $n \geq 1$, by a constant $C_0>0$ say,
then according to the proof of \cite[Theorem 8.5.2]{HVVW}, we
will obtain that the operators
$n^{\alpha}(\rho T)^{n-1}\prod\limits_{j=1}^N(I-\overline{\xi_j} \rho T)^\alpha$ 
belong to the strong closure of 
the absolute convex hull of $2C_0F$, which is 
$\mathcal{R}$-bounded by \cite[Propositions 8.1.21 and 8.1.22]{HVVW}.

Thus it suffices to show that the integrals
$$
I_n = n^\alpha \int_{\partial E_s} 
\abs{\lambda}^{n}\prod\limits_{j=1}^N \abs{1-\overline{\xi_j}\lambda}^{\alpha-1} 
\abs{\dx \lambda}
$$
are uniformly bounded for $n \geq 1$.

To prove this, let us decompose each of them as $I_n = \sum\limits_{k=1}^N I_{n,k,0}+I_{n,k,1}+I_{n,k,2}$, where 
$$
I_{n,k,0} = n^\alpha \int_{\Gamma_{k}} \abs{\lambda}^{n}\prod\limits_{j=1}^N \abs{1-\overline{\xi_j}\lambda}^{\alpha-1} \abs{\dx \lambda};
$$
$$
I_{n,k,1} = n^\alpha \int_{\Gamma_{k,1}} \abs{\lambda}^{n}\prod\limits_{j=1}^N \abs{1-\overline{\xi_j}\lambda}^{\alpha-1} \abs{\dx \lambda};
$$
$$
I_{n,k,2} = n^{\alpha}\int_{\Gamma_{k,2}} \abs{\lambda}^{n}\prod\limits_{j=1}^N \abs{1-\overline{\xi_j}
\lambda}^{\alpha-1} \abs{\dx \lambda}.
$$
Let us first consider the integrals $I_{n,k,0}$. 
For all $k \in \{1,\ldots,N\}$ and for any $\lambda \in \Gamma_k$, 
we both have 
\begin{equation}
\sin \theta \leq \abs{1-\overline{\xi_k}\lambda} \leq 2 
\quad\textrm{ and }\quad \abs{\lambda} = s
\end{equation}
Since the sequence $(n^\alpha s^n)_{n\geq 1}$ is bounded, 
the sequences $(I_{n,k,0})_{n\geq 1}$ are bounded as well.

Let us now consider  the integrals $I_{n,k,1}$.  
We set $\gamma = \frac \pi 2 - \theta$. The segment $\Gamma_{k,1}$ joins $\xi_k$ to $\xi_k(1 - \cos\gamma e^{-i\gamma})$, for all $k \in \{1,\ldots,N\}$. For any $t \in [0,\cos\gamma]$, we have $t^2 \leq t \cos\gamma$, and  
$$
\abs{1-te^{-i\gamma}}^2 = 1 + t^2 - 2t\cos\gamma \leq 1 - t\cos\gamma.
$$
Hence 
\begin{align*}
I_{n,k,1} & = n^\alpha \int_{0}^{\cos\gamma} \abs{1-te^{-i\gamma}}^{n}\prod\limits_{j=1}^N\abs{1-\overline{\xi_j}
\xi_k(1-te^{-i\gamma})}^{\alpha-1} \dx t\\
&\leq n^\alpha \int_{0}^{\cos\gamma} (1-t\cos\gamma)^{\frac{n}{2}}
t^{\alpha-1} \underset{j\neq k}{\prod\limits_{j=1}^N}
\abs{1-\overline{\xi_j}\xi_k(1-te^{-i\gamma})}^{\alpha-1} \dx t\\
&\leq C n^\alpha \int_{0}^{\cos\gamma} 
(1-t\cos\gamma)^{\frac{n}{2}} t^{\alpha-1} \dx t,
\end{align*}
where $C = \displaystyle{\sup\limits_{t \in 
[0,\cos\gamma]} \underset{j\neq k}{\prod\limits_{j=1}^N}
\abs{1-\overline{\xi_j}\xi_k(1-te^{-i\gamma})}^{\alpha-1}}$.

Changing $t$ into $s = t\cos\gamma$ and using the inequality $1-s \leq e^{-s}$, we deduce that 
$$
I_{n,k,1} \leq C\frac {n^\alpha} {(\cos\gamma)^\alpha} \int_{0}^{\cos^2\gamma} s^{\alpha-1}e^{-\frac{sn}2} \dx s.
$$
Changing now $s$ into $u = \cfrac {sn} 2$, we have
 \begin{align*}
I_{n,k,1} \leq C \frac {2^\alpha} 
{(\cos\gamma)^\alpha} \int_0^\infty u^{\alpha-1}e^{-u}\dx u.
\end{align*}
Hence the sequences $(I_{n,k,1})_{n\geq 1}$ are bounded, 
for all $k \in \{1,\ldots,N\}$. This is also the case 
of the sequences $(I_{n,k,2})_{n\geq 1}$, 
because $I_{n,k,1} = I_{n,k,2}$. Hence the sequence $(I_n)_{n\geq 1}$ is bounded.
\end{proof}

    %\begin{lemma}\label{k^m}
    %    Let $T$ be a Ritt$_E$ operator. There exists $K>0$ such that for any $\rho \in (0,1]$, any $k \geq 1$ and any $m \geq 1$ : 
    %    \begin{equation}\label{rhoritt}
    %        \bignorm{T^k\prod\limits_{j=1}^N(I-\overline{\xi_j}\rho T)^m} \leq \frac K {k^m}
    %    \end{equation}
    %\end{lemma}
    
    %\begin{proof}
    %    If $T$ is Ritt$_E$, then $\rho T$ is Ritt$_E$ as well (see []). Then, according to [], theorem (), there exists $C > 0$ (not depending on $\rho$) such that, for any $k \geq 1$ : 
    %    \begin{equation*}
    %        \bignorm{T^k\prod\limits_{j=1}^N(I-\overline{\xi_j}\rho T)} \leq \frac C {k}
    %    \end{equation*}
    %    Raising the operator in the norm to the power of $m \geq 1$, we get : 
    %    \begin{equation*}
    %        \bignorm{T^{mk}\prod\limits_{j=1}^N(I-\overline{\xi_j}\rho T)^m} \leq \frac C {k^m}
    %    \end{equation*}
    %    Fix $k \geq 1$ and set $q \geq 0$ and $0 \leq r < m$ such that $k = mq + r$. Then, by \ref{rhoritt}, we have : 
    %    \begin{align*}
    %        \bignorm{T^{k}\prod\limits_{j=1}^N(I-\overline{\xi_j}\rho T)^m} &\leq \bignorm{T^{qm + r}\prod\limits_{j=1}^N(I-\overline{\xi_j}\rho T)^m}\\
    %        &\leq \bignorm{T}^r \bignorm{T^{qm}\prod\limits_{j=1}^N(I-\overline{\xi_j}\rho T)^m}\\ 
    %        &\leq \frac {C \bignorm{T}^r } {q^m}\\
    %        &\lesssim \frac 1 {k^m}
    %    \end{align*}
    %\end{proof}

We now deal with approximation of square functions (see Definition \ref{SF}).

\begin{lemma}\label{SQlimrho1}
Consider an operator $T : X \rightarrow X$ and assume that
$T$ is $\rittE$.  Then for all $x \in \ran\left(\prod\limits_{j=1}^N(I-\overline{\xi_j}T)\right)$ 
and all $\alpha > 0$, we have $\norm{x}_{T, \alpha}<\infty$ and
$\norm{x}_{\rho T, \alpha} \rightarrow \norm{x}_{T, \alpha}$,
when $\rho \in (0,1)$ tends to $1$.
\end{lemma}

\begin{proof}
Let $x \in \ran\left(\prod\limits_{j=1}^N(I-\overline{\xi_j}T)\right)$ and let $x' \in X$ such that $x = \prod\limits_{j=1}^N(I-\overline{\xi_j}T) x'$. We set $x_\rho = \prod\limits_{j=1}^N(I-\overline{\xi_j}\rho T)^{-1}x$ for all $\rho \in (0,1)$. This is well-defined (each 
$I-\overline{\xi_j}\rho T$ is invertible) and we have
$$
x_\rho = \prod\limits_{j=1}^N(I-\overline{\xi_j}T)(I-\overline{\xi_j}\rho T)^{-1} x'.
$$
We write
\begin{align*}
\prod\limits_{j=1}^N(I-\overline{\xi_j}T)(I-\overline{\xi_j}\rho T)^{-1} 
&= \prod\limits_{j=1}^N[I-\overline{\xi_j}(1-\rho)(I-\overline{\xi_j}\rho T)^{-1}T]\\
&= \prod\limits_{j=1}^N[I-\overline{\xi_j}(\rho^{-1}-1)R(\rho^{-1},\overline{\xi_j}T)T].
\end{align*}
Since $T$ is a Ritt$_E$ operator, there exists $c>0$ such that for any $j_0 \in \{1,\ldots,N\}$ and 
any $\rho\in (0,1)$, we have
\begin{align*}
\bignorm{(\rho^{-1}-1)R(\rho^{-1},\overline{\xi_{j_0}}T)} 
&= \bignorm{(\rho^{-1}-1)R(\rho^{-1}\xi_{j_0},T)}\\
&\leq (\rho^{-1}-1)\cfrac{c}{\prod\limits_{j=1}^N\abs{\xi_j-\rho^{-1}\xi_{j_0}}}\\
&= \cfrac{c}{\prod\limits_{j=1, j\neq j_0}^N \abs{\xi_j-\rho^{-1}\xi_{j_0}}}\\
&\leq \cfrac{c}{\prod\limits_{j=1, j\neq j_0}^N \abs{\xi_j-\xi_{j_0}}}.
\end{align*}
We deduce that the operators
$$
I-\overline{\xi_j}(\rho^{-1}-1)R(\rho^{-1},\overline{\xi_j}T)T
$$
are uniformly bounded,
when $\rho \in (0,1)$ and $j \in \{1,\ldots,N\}$. 
Hence there exists a constant $C_1>0$ such that
$$
\norm{x_{\rho}} \leq C_1,\qquad\rho\in (0,1).
$$

Let $r \in (0,1)$ such that $T$ is of type $r$ and let $s \in (r,1)$. Then for any $\rho\in(0,1)$ and any $k \geq 1$ we can write, using (\ref{FormulaContour}) with $\alpha +1 $ instead of $\alpha$,
$$
k^{\alpha+1} (\rho T)^{k-1} \prod\limits_{j=1}^N 
(I - \overline{\xi_j}\rho T)^{\alpha+1} = \frac {k^{\alpha+1}} {2\pi i} 
\int_{\partial E_s} \lambda^{k-1} \prod\limits_{j=1}^N 
(1 - \overline{\xi_j} \lambda)^\alpha \left[\prod\limits_{j=1}^N (1 - \overline{\xi_j} \lambda) R(\lambda,\rho T)\right] \dx \lambda.
$$

It follows from Lemma \ref{indRittRho} that the set
$$
\Bigl\{\prod\limits_{j=1}^N (1 - \overline{\xi_j} \lambda) R(\lambda,\rho T)\, :\,
\lambda\in\partial E_s\setminus E,\ \rho\in (0,1)\Bigr\}.
$$
is bounded.
Hence there exists a constant $C$ independent of 
$\rho\in (0,1)$ and $k\geq 1$ such that
$$
k^{\alpha + 1} \Bignorm{(\rho T)^{k-1} \prod\limits_{j=1}^N (I - \overline{\xi_j}\rho T)^{\alpha+1}} 
\leq C\, \frac {k^{\alpha+1}} {2\pi} \int_{\partial E_s} \abs{\lambda}^{k-1}  
\prod\limits_{j=1}^N \vert 1-\overline{\xi_j}\lambda\vert^\alpha\vert\dx{\lambda}\vert.
$$
According to the proof of Proposition \ref{Rbddalpha}, the right 
handside of this inequality is uniformly bounded, for $k\geq 1$.
We deduce 
that there exists a constant $C_2 >0$ independent of 
$\rho\in (0,1)$ and $k\geq 1$ such that
$$
k^{\alpha + 1} \Bignorm{(\rho T)^{k-1} \prod\limits_{j=1}^N
(I - \overline{\xi_j}\rho T)^{\alpha+1}} \leq C_2.
$$

Now, we write
\begin{align*}
\sum\limits_{k=1}^\infty k^{\alpha-1/2} \sup\limits_{\rho \in (0,1)} 
\bignorm{(\rho T)^{k-1} &\prod\limits_{j=1}^N(I-\overline{\xi_j}\rho T)^\alpha x} \\
&= \sum\limits_{k=1}^\infty k^{\alpha-1/2} \sup\limits_{\rho \in (0,1)} \bignorm{(\rho T)^{k-1}\prod\limits_{j=1}^N(I-\overline{\xi_j}\rho T)^{\alpha+1} x_\rho}\\
&= \sum\limits_{k=1}^\infty k^{-\frac{3}{2}}\left(k^{\alpha+1} \sup\limits_{\rho \in (0,1)} \bignorm{(\rho T)^{k-1}\prod\limits_{j=1}^N(I-\overline{\xi_j}\rho T)^{\alpha+1} x_\rho}\right)\\
&\leq C_1C_2\sum\limits_{k=1}^\infty k^{-\frac{3}{2}} < \infty.
\end{align*}
Since 
$$(\rho T)^{k-1} (I - \rho \overline{\xi_j}\rho T)^\alpha x 
\underset{\rho \rightarrow  1}{\longrightarrow} T^{k-1} 
(I - \overline{\xi_j}T)^\alpha x,
$$
for all $k \geq 1$, then 
by the above estimate, we derive that 
$\norm{x}_{\rho T, \alpha} \underset{\rho \rightarrow  1}{\longrightarrow} \norm{x}_{T, \alpha}$
and that the latter is finite.
\end{proof}

We end this section with a property of stability for the notion of $\mathcal{R}$-Ritt$_E$ under the assumption of $K$-convexity.

\begin{theorem}
    Let $X$ be a $K$-convex Banach space and $T : X \to X$ a $\mathcal{R}$-Ritt$_E$ operator. Then $T^*$ is also a $\mathcal{R}$-Ritt$_E$ operator.
\end{theorem}

\begin{proof}
    By definition of $\mathcal{R}$-Ritt$_E$ operators, the set $F = \{\prod\limits_{j=1}^N(I - \overline{\xi_j}T) R(z,T) : 1 < \abs{z} < 2\}$ is $\mathcal{R}$-bounded.
    
    Thus, the set $F^* = \{S^*, S \in F\}$, equal to $\{\prod\limits_{j=1}^N (I - \overline{\xi_j}T^*) R(z,T^*) : 1 < \abs{z} < 2\}$, is $\mathcal{R}$-bounded as well, by Theorem \ref{KcvRbdd}. 
    Hence, $T^*$ is $\mathcal{R}$-Ritt$_E$.
\end{proof}

\section{Franks-McIntosh decomposition on $E_r$ domains}
\label{FMI}

Let $r \in (0,1)$ be $E$-large enough and let $s\in(r,1)$.
The aim of this section is to exhibit a countable family $(\Phi_i)_i$ 
of functions in $H_0^\infty(E_{r})$, such that 
any function $h\in H^\infty(E_{s})$ can be written as an infinite sum
$$
h = \sum\limits_{i} \alpha_i \Phi_i,
$$
for some bounded family $(\alpha_i)_i$ of complex numbers. 
We will also prove such a decomposition property for vector valued 
functions $h$. This construction is an adaptation of the 
Franks-McIntosh decomposition from \cite{FMI}, 
for domains $E_r$. It is actually implicit in 
\cite{FMI} but no actual proof is written there.
This section is also a generalization of \cite[Section 6]{ArLM0}.

\subsection{Construction of the family $(\Phi_i)_i$}
We assume that $r$ is $E$-large enough and fix some $s\in(r,1)$. 
In this case, $(\partial E_r \cap \partial E_s) \setminus E = \emptyset$,
and the boundaries $\partial E_r$ and $\partial E_s$ 
are composed of segments of the type $[\xi_j,.]$ and of
arcs of circle joining two consecutive segments 
(see Figure \ref{fig2} below). We will build a path 
connecting all the $\xi_j$ and
located between $\partial E_r$ and $\partial E_s$.

We set $\xi_{N+1} = \xi_1$.
We fix some $r'\in(r,s)$, as well as some $\theta_0>0$ and $\delta>0$. 
For any $j \in \{1,\ldots,N\}$, we  define the oriented segments, 
$$
\Gamma_{1,j} = [\xi_j,r' e^{i\theta_0}\xi_j] \qquad \hbox{and}\qquad
\Gamma_{2,j} = [r' e^{-i\theta_0}\xi_{j+1},\xi_{j+1}].
$$ 
Choosing $\theta_0$ small enough, we        
can consider, for any $j \in \{1,\ldots,N\}$, a finite family of segments 
$\{\gamma_{0,j,k}\}_{k=0}^{M_j}$, whose concatenation $\Gamma_{0,j}$
is a path going from
$r'e^{i\theta_0}\xi_j$ to $r'e^{-i\theta_0}\xi_{j+1}$ counterclockwise, without
meeting either $\partial E_r$ or $\partial E_s$. For any 
$j \in \{1,\ldots,N\}$ and $k\in \{0,\ldots,M_j\}$, we let
$z_{0,j,k}$ denote the center of $\gamma_{0,j,k}$ and
we let $D_{0,j,k}$ denote the  open disc of centre $z_{0,j,k}$ and
radius $\delta.$ We may choose $\delta$ small enough to ensure that the discs
$D_{0,j,k}$ do not intersect $\partial E_r$. We may also assume
that the segments $\gamma_{0,j,k}$ have length $\leq\frac{\delta}{2}$.

\smallskip
\begin{center}
\includegraphics[scale=0.3]{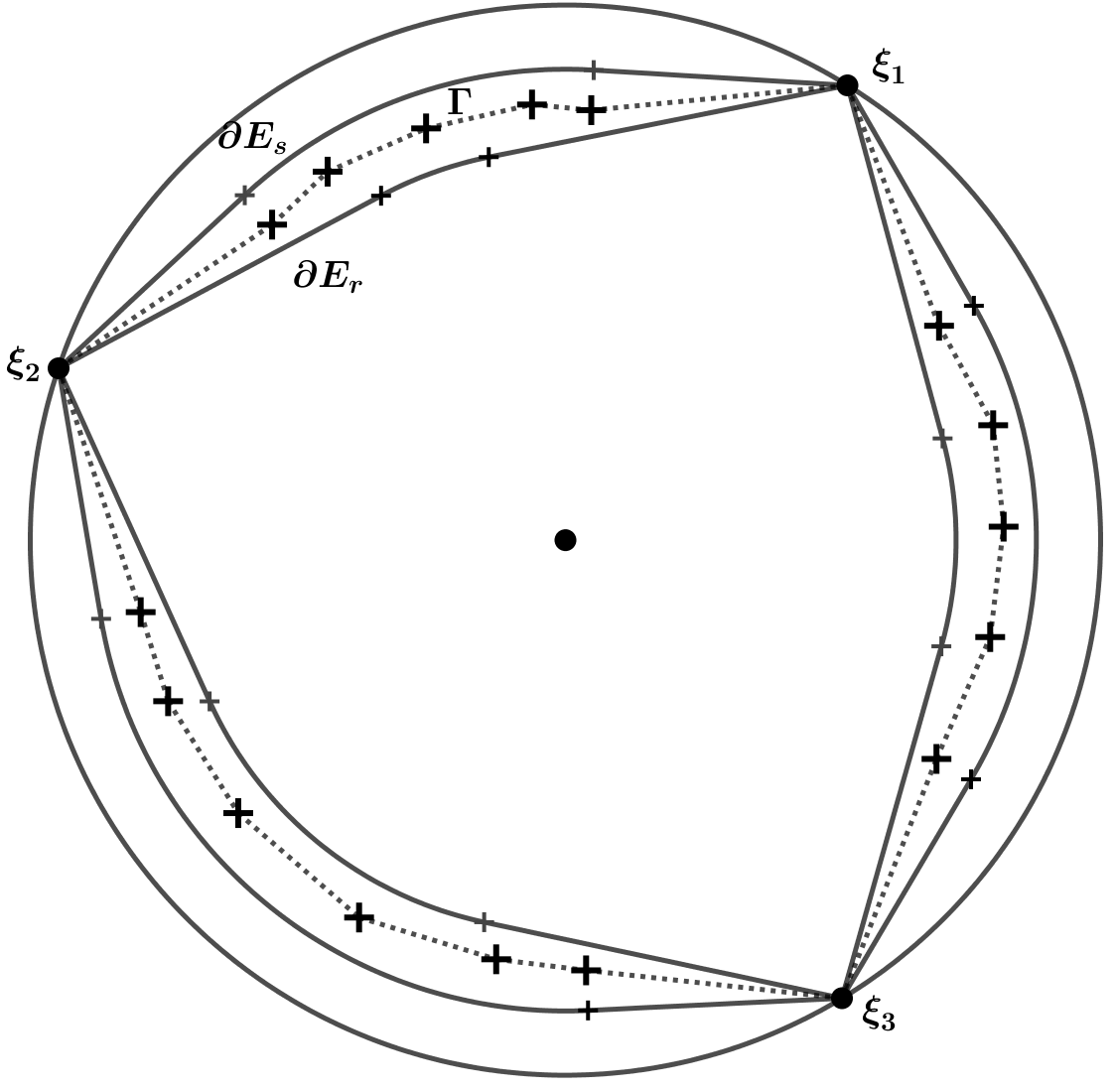}
\captionof{figure}{Illustration of the contour with $N = 3$}
\label{fig2}
\end{center}
\smallskip
Finally we let $\Gamma$ be the concatenation of the (finite) family,
$$
\bigl\{\Gamma_{m,j}\,
:\, m \in \{0,1,2\},\, j\in\{1,\ldots,N\}\bigr\},
$$
oriented counterclockwise. By construction, $\Gamma$ is a simple
closed contour containing $E_r$  in its interior.

Note that all the segments $\Gamma_{1,j}$ and $\Gamma_{2,j}$ 
have the same length. Let $l$ denote this length. Let $\rho>1$.
For any $j \in \{1,\ldots,N\}$, we consider the segments
$$
\gamma_{1,j,k} = \bigl\{z \in \Gamma_{1,j} : l\rho^{-k-1} \leq     
\abs{z-\xi_j} \leq l\rho^{-k}\bigr\} , \qquad k\geq 0,
$$      
whose union is equal to $\Gamma_{1,j}\setminus\{\xi_j\}$.
We let $z_{1,j,k}$ denote the center of $\gamma_{1,j,k}$ and we let
$D_{1,j,k}$ be the open disc of centre 
$z_{1,j,k}$ and radius $s_{k} = l(\rho^{-k}-\rho^{-k-1})$.
We choose $\rho$ close enough to $1$ to ensure that for all $j \in \{1,\ldots,N\}$
and $k\geq 0$, $D_{1,j,k}$ does not intersect $\partial{E_{r}}$.
Next we define, analogously, segments 
$$
\gamma_{2,j,k} = \bigl\{z \in \Gamma_{2,j} : l\rho^{-k-1} \leq     
\abs{z-\xi_j} \leq l\rho^{-k}\bigr\} 
$$
as well as $z_{2,j,k}$ and
$D_{2,j,k}$ for all $j \in \{1,\ldots,N\}$ and $k\geq 0$.

We now set, for all $\xi \in E_{r}$ and  $z$ in the union of all the discs 
$D_{0,j,k}$, $D_{1,j,k}$ and $D_{2,j,k}$,
$$
K(z,\xi) = \cfrac {\prod\limits_{j=1}^N
(1-\overline{\xi_j}z)^{1/2}(1-\overline{\xi_j}\xi)^{1/2}}{z-\xi}\,.
$$
We will need estimates on this function, which will be based on the
following.

\begin{lemma}\label{Unif}
For $\xi\in E_r$ and $z$ in the union of all the discs 
$D_{0,j,k}$, $D_{1,j,k}$ and $D_{2,j,k}$, we have uniform estimates 
$$
\abs{1-\overline{\xi_j}z}\lesssim \abs{z-\xi}
\qquad\hbox{and}\qquad
\abs{1-\overline{\xi_j}\xi}\lesssim \abs{z-\xi}.
$$
\end{lemma}

\begin{proof}
Let $z \in D_{1,j,k}$ for some $j \in \{1,..,N\}$ and $k \geq 0$, 
and let $\xi \in E_r$. It follows from the above construction
that there exist $\nu\in\bigl(\frac{\pi}{2},\pi\bigr)$
and $\theta\in\bigl(0,\pi-\nu\bigr)$, not depending
on either $z$ or $\xi$, such that 
$$
{\rm Arg}(\xi\overline{\xi_j}-1)\geq  \nu+\theta
\qquad\hbox{and}\qquad
{\rm Arg}(z\overline{\xi_j}-1)\leq \nu.
$$

\smallskip
\begin{center}
\includegraphics[scale=0.4]{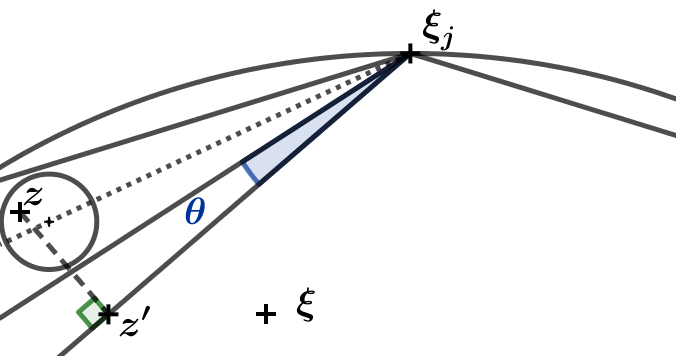}
%\captionof{figure}{}
\label{fig3}
\end{center}

\smallskip
Let $z'$ be the projection of $z$ on the line generated 
by the segment $\Gamma_{1,j}$. We have the inequality
$$
\cfrac {\abs{z'-z}}{\abs{\xi_j-z}} \geq \sin(\theta).
$$
Since $\abs{z-\xi} \geq \abs{z'-\xi}$, we deduce 
$$
\abs{\xi_j-z} \leq \cfrac{1}{\sin(\theta)}\abs{\xi-z}
$$
Switching the roles of $\xi$ and $z$, we obtain as well that 
$\abs{\xi_j-\xi} \leq \cfrac{1}{\sin(\theta)}\abs{\xi-z}$.
These estimates are also true when $z \in 
D_{2,j,k}$ for some $j \in \{1,..,N\}$ and $k \geq 0$. 

Finally, $\abs{\xi-z}$ is bounded away from $0$
when $\xi\in E_r$ and  $z$ in the union of all the discs 
$D_{0,j,k}$. The result follows.
\end{proof}

Let $j\in\{1,\ldots,N\}$.
Observe that if $\xi \in E_{r}$ is
sufficiently close to the vertex $\xi_{j}$, then
there exists $q \in \N$ such that 
\begin{equation}\label{estimateq}
l\rho^{-q-1} \leq \abs{\xi_{j}-\xi} \leq l\rho^{-q}.
\end{equation}
We let $B_{j,q}$ be the set of all these $\xi$. 

\smallskip
\begin{center}
\includegraphics[scale=0.4]{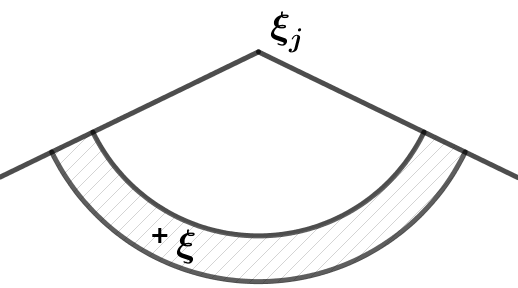}
\captionof{figure}{Sectorial band of little 
radius $l\rho^{-q-1}$ and of big radius $l\rho^{-q}$}
%\label{fig1}
\end{center}

\smallskip
Further we 
set 
$$
A_j=E_r\setminus\bigcup_{q\geq 0} B_{j,q}.
$$

\begin{lemma}\label{K} Let $j\in\{1,\ldots,N\}$.
We have the following uniform estimates.
\begin{itemize}
\item [(a)] For $\xi\in B_{j,q}$ and $z$ 
in the union of all the discs 
$D_{1,j,k}$ and $D_{2,j,k}$ (for $k\geq 0$),
$$
\abs{K(z,\xi)} \lesssim \rho^{-\frac{\vert k-q\vert}{2}}.
$$
\item [(b)] For $\xi\in A_j$ and $z$ 
in the union of all the discs 
$D_{1,j,k}$ and $D_{2,j,k}$ (for $k\geq 0$),
$$
\abs{K(z,\xi)}\lesssim \rho^{-\frac{k}{2}}.
$$
\item [(c)] For $\xi\in E_r$ and $z$ 
in the union of all the discs 
$D_{0,j,k}$ (for $k\geq 0$), $\abs{K(z,\xi)}\lesssim 1$.
\end{itemize}
\end{lemma}

\begin{proof}
Consider $\xi\in A_j$ and $z$ 
in the union of all the discs 
$D_{1,j,k}$ and $D_{2,j,k}$. Then $\abs{\xi-z}$ is bounded away from $0$, hence
$$
\abs{K(z,\xi)}\lesssim 
\abs{\xi_{j}-z}^{1/2}\lesssim \rho^{-\frac{k}{2}}.
$$
This proves (b).

Now consider $\xi\in B_{j,q}$ for some $q\geq 0$
and $z$ 
in the union of all the discs 
$D_{1,j,k}$ and $D_{2,j,k}$. Then using Lemma \ref{Unif}
and (\ref{estimateq}), we have
\begin{align*}
\abs{K(z,\xi)} &= \frac {\prod\limits_{d=1}^N
\abs{1-\bar{\xi_d}z}^{1/2}\abs{1-\bar{\xi_d}\xi}^{1/2}}{\abs{z-\xi}}\\
&= \frac{\abs{\xi_{j}-z}^{1/2}\abs{\xi_{j}-\xi}^{1/2}} 
{\abs{z-\xi}}
\prod_{d\neq j}\abs{\xi_d-z}^{1/2}\abs{\xi_d-\xi}^{1/2} \\
&\lesssim \frac {\rho^{-k/2}\rho^{-q/2}}{\abs{z-\xi}}\\
&\lesssim \frac {\rho^{-k/2}\rho^{-q/2}}{\rho^{-\min(k,q)}}
= \rho^{\min(k,q)-\frac{q}{2}-\frac{k}{2}}= \rho^{\frac {-\abs{k-q}}{2}}.
\end{align*}
This proves (a).

For $\xi\in E_r$ and $z$ 
in the union of all the discs 
$D_{0,j,k}$ (for $k\geq 0$), $\abs{\xi-z}$ is bounded away from $0$.
The assertion (c) follows at once.
\end{proof}

We consider now the Hilbert spaces 
$$
H_{m,j,k} = 
L^2\left(\gamma_{m,j,k},\left\vert\frac {\dx{z}}{\xi_{j}-z}\right\vert\right),
$$
for 
$m \in \{1,2\}$, 
$j \in \{1,\ldots,N\}$ and $k\geq 0$. For $m=0$, 
we consider 
$$
H_{0,j,k} = L^2\left(\gamma_{0,j,k},\abs{\dx{z}}\right),
$$
where 
$ j \in \{1,\dots,N\}$ and 
$k \in \{0,..,M_j\}$.

We first observe that the underlying measure spaces
on which these $L^2$-spaces are built
are uniformly bounded. Indeed if $C$ is the greatest length of the segments $\{\gamma_{0,j,k}\}$, for all $j\in \{1,\ldots,N\}$ and all 
$k \in \{0,..,M_j\}$, then
$$  
\int_{\gamma_{0,j,k}} \abs{\dx{z}} \leq C
$$
Further if $m\in\{1,2\}$, then a straightforward computation yields
\begin{equation}\label{logrho}
\int_{\gamma_{m,j,k}} \abs{\frac{\dx{z}}{\xi_{j_0}-z}} = \log(\rho).
\end{equation}

For any $m,j,k$, we consider a Hilbertian
basis $\{e_{m,j,k,p}\}_{p=0}^\infty$ of $H_{m,j,k}$ 
such that for all $n\geq 0$,
the linear space spanned by 
$\{e_{m,j,k,p}\}_{p=0}^n$ is equal to the space of polynomial 
functions of degree less than or equal to $n$.
Next we define, for all $m,j,k,p$, a function
$\Phi_{m,j,k,p}$ defined on $E_r$ by
$$
\Phi_{m,j,k,p} (\xi) = \frac 1 {2\pi i}  
\int_{\gamma_{m,j,k}} 
\overline{e_{m,j,k,p}(z)} K(z,\xi) 
\frac {\dx{z}}{\prod\limits_{d=1}^N(1 -\overline{\xi_{d}}z)}.
$$
This is a well-defined holomorphic function. Moreover by the 
Cauchy-Schwarz inequality and the definition
of $K$ we have, for $\xi \in E_r$,
\begin{align*}
\abs{\Phi_{m,j,k,p}(\xi)} &\leq \frac 1 {2\pi} 
\left({\int_{\gamma_{m,j,k}} 
\abs{e_{m,j,k,p}(z)}^2\frac{\abs{\dx{z}}}{\abs{\xi_{j}-z}}}
\right)^{1/2}\left(\int_{\gamma_{m,j,k}} 
\abs{\frac{K(z,\xi)}{\prod\limits_{d\not=j}(\xi_d-z)}}^2
\frac{\abs{\dx{z}}}{\abs{\xi_{j}-z}}\right)^{1/2}\\
&\leq \frac{\prod\limits_{d=1}^N \abs{\xi_d-\xi}^{1/2}}{2\pi}
\left({\int_{\gamma_{m,j,k}}
\frac{1}{\prod\limits_{d\not=j}
\vert \xi_d-z\vert}\,\frac{\abs{\dx{z}}}{\abs{z-\xi}^2}}
\right)^{1/2}\\
&\lesssim \prod\limits_{d=1}^N \abs{\xi_d-\xi}^{1/2}.
\end{align*}
Indeed, $\abs{z-\xi}$ is bounded away from $0$ when 
$z$ varies in 
$\gamma_{m,j,k}$ and $\xi$ varies in
$E_r$. Hence the integral
in the right-hand side is bounded independently of $\xi\in E_r$
Thus, 
\begin{equation}\label{HI0}
\Phi_{m,j,k,p} \in \HI_0(E_r).
\end{equation}

We will now establish much more accurate estimates on $\Phi_{m,j,k,p}$, 
given by the next lemma.

 %LEMME ESTIMATIONS SUR PHI
\begin{lemma}\label{estimationsPHI}
Let $j\in\{1,\ldots,N\}$. 
There exists $c>0$ such that:
\begin{itemize}
\item [(a)]
For all $m\in\{1,2\}$, $k\geq 0$, $p\geq 0$,
$q\geq 0$ and $\xi\in B_{j,q}$,
$$
\abs{\Phi_{m,j,k,p}(\xi)} \leq c2^{-p}
\rho^{-\frac{\abs{k-q}}{2}}.
$$
\item [(b)]
For all $m\in\{1,2\}$, $k\geq 0$, $p\geq 0$,
and $\xi\in A_{j}$,
$$
\abs{\Phi_{m,j,k,p}(\xi)} \leq c2^{-p}
\rho^{-\frac{k}{2}}.
$$
\item [(c)] 
For all $k\geq 0$, $p\geq 0$ 
and $\xi\in E_r$,
$$
\abs{\Phi_{0,j,k,p}(\xi)} \leq c2^{-p}.
$$
\end{itemize}
\end{lemma}

%PREUVE DU LEMME

\begin{proof}
We consider $m\in\{1,2\}$, $k\geq 0$, $p\geq 0$,
$q\geq 0$ and $\xi\in B_{j,q}$ as in (a).
The restriction of $K(.,\xi)$ to $D_{m,j,k}$ is analytic.
Hence we can consider the 
power series expansion of the function 
$\displaystyle{z \mapsto \frac{K(z,\xi)}{\prod\limits_{d\not=j}
(1-\overline{\xi_{d}}z)}}$ on $D_{m,j,k}$, that we write as
$$
\displaystyle{\frac{K(z,\xi)}{\prod\limits_{d\not=j}
(1-\overline{\xi_{d}}z)}\, = \, \sum\limits_{n=0}^\infty b_{m,j,k,n}
\left(\frac{z-z_{m,j,k}}{s_k}\right)^n},
\qquad z\in D_{m,j,k}.
$$
One can write
$$
\displaystyle{\Phi_{m,j,k,p}(\xi) = \frac {1}{2\pi i} 
\int_{\gamma_{m,j,k}}\overline{e_{m,j,k,p}(z)}\left(
\sum\limits_{n=0}^\infty b_{m,j,k,n}
\left(\frac{z-z_{m,j,k}}{s_k}\right)^n\right) 
\frac{\dx{z}}{1-\overline{\xi_{j}}z}}.
$$
In the Hilbert space $H_{m,j,k}$, 
each $e_{m,j,k,p}$ is orthogonal to all polynomial functions of 
degree $n<p$. Moreover the measures 
$\displaystyle{\abs{\frac{\dx{z}}{\xi_{j_0}-z}}}$ 
and $\displaystyle{\frac{\dx{z}}{1-\overline{\xi_{j}}z}}$
are proportional. Hence we actually have
\begin{equation}\label{Develop}
\Phi_{m,j,k,p}(\xi) = 
\frac 1 {2\pi i}
\int_{\gamma_{m,j,k}}\overline{e_{m,j,k,p}(z)} 
\left(\sum\limits_{n=p}^\infty b_{m,j,k,n} 
\left(\frac{z-z_{m,j,k}}{s_k}\right)^n\right) 
\frac{\dx{z}}{1-\overline{\xi_{j}}z}.
\end{equation}

To estimate $\Phi_{m,j,k,p}(\xi)$, we will now estimate the sum 
$\sum_{n=p}^\infty\,$ in the above integral. 
The Bessel-Parseval identity in $L^2\left(\partial 
D_{m,j_0,k}, \frac{\abs{\dx{z}}}{2\pi s_k}\right)$ yields
\begin{align*}
\left(\sum_{n=0}^\infty \abs{b_{m,j,k,p}}^2\right)^{1/2} &=
\biggnorm{z\mapsto \frac 
{K(z,\xi)}{\prod\limits_{d\not=j}^N(1-\overline{\xi_{d}}z)}}_{L^2}\\
&\leq\sup\Bigl\{\abs{K(z,\xi)}
\prod\limits_{d\not=j}\bigl\vert
1-\overline{\xi_{d}}z\bigr\vert^{-1} \,:\ z \in D_{m,j,k}\Bigr\}.
\end{align*}
Applying Lemma \ref{K}, (a), we deduce an estimate
$$
\left(\sum_{n=0}^\infty \abs{b_{m,j,k,p}}^2\right)^{1/2}
\lesssim \rho^{-\frac{\vert k-q\vert}{2}}.
$$
Let $z\in\gamma_{m,j,k}$.
By the Cauchy-Schwarz inequality, the above estimate and
the fact that $\vert z -z_{m,j,k}\vert\leq s_k/2$, we
obtain
\begin{align*}
\sum_{n=p}^\infty \abs{b_{m,j,k,n}
\left(\frac{z-z_{m,j,k}}{s_k}\right)^n} 
&\leq \left(\sum_{n=p}^\infty 
\abs{b_{m,j,k,n}}^2\right)^{1/2}
\left(\sum_{n=p}^\infty 
\abs{\frac{z-z_{m,j,k}}{s_k}}^{2n}\right)^{1/2}\\
&\leq \left(\sum_{n=0}^\infty 
\abs{b_{m,j,k,n}}^2\right)^{1/2}
\left(\sum_{n=p}^\infty \frac{1}{{2}^{n}}\right)^{1/2}\\
&\lesssim \rho^{-\frac{\abs{k-q}}{2}}  2^{-p}.
\end{align*}

By (\ref{Develop}), $\vert \Phi_{m,j,k,p}(\xi) \vert$ is less than or
equal to
$$
\frac{1}{2\pi} 
\left(\int_{\gamma_{m,j,k}}\vert e_{m,j,k,p}(z)\vert^2
\Bigl\vert\frac{\dx{z}}{1-\overline{\xi_{j}}z}\Bigr\vert\right)^{1/2}
\left(\int_{\gamma_{m,j,k}}\biggl\vert 
\sum\limits_{n=p}^\infty b_{m,j,k,n} 
\left(\frac{z-z_{m,j,k}}{s_k}\right)^n \biggr\vert^2
\Bigl\vert\frac{\dx{z}}{1-\overline{\xi_{j}}z}\Bigr\vert\right)^{1/2}
$$
We deduce $\vert \Phi_{m,j,k,p}(\xi) \vert\lesssim 
\rho^{-\frac{\abs{k-q}}{2}} 2^{-p}$, which proves (a).

The proofs of (b) and (c) are similar, using the corresponding
parts of Lemma \ref{K}.
\end{proof}

The following is a straightforward consequence of Lemma \ref{estimationsPHI}.

\begin{proposition}\label{CV}
There exists a constant $C>0$ such that for all $\xi\in E_r$,
$$
\sum_{m,j,k,p}\abs{\Phi_{m,j,k,p}(\xi)}^{1/2} \leq C.
$$
\end{proposition}

 %DECOMPOSITION DE FRANKS MCINTOSH
 \subsection{Decomposition of Franks-McIntosh on $E_{r}$}
	 
\begin{proposition}\label{Alpha}
Let $h \in H^\infty(E_{s})$. For any $m\in\{1,2\}$, $j\in\{1,\ldots,N\}$,
$k\geq 0$ and $p\geq 0$, set
$$
\alpha_{m,j,k,p} = \int_{\gamma_{m,j,k}}
h(z){e_{m,j,k,p}(z)} \abs{\frac{\dx{z}}{\xi_{j}-z}}.
$$
Next, for any $j\in\{1,\ldots,N\}$,
$k\in\{0,\ldots,M_j\}$ and $p\geq 0$, set
$$
\alpha_{0,j,k,p} = \int_{\gamma_{0,j,k}}
h(z){e_{m,j,k,p}(z)} \abs{\dx{z}}.
$$
Then $\{\alpha_{m,j,k,p}\}$ is a well-defined family
of complex numbers, this family is bounded
and for all $\xi \in E_{r}$ : 
$$
\displaystyle{h(\xi) = \sum\limits_{m,j,k,p} \alpha_{m,j,k,p}\Phi_{m,j,k,p}(\xi)}
$$
\end{proposition}

%PREUVE 
\begin{proof}
Let $m \in \{1,2\}$, $j \in \{1,\ldots,N\}$, $k \geq 0$ and $p \geq 0$.
According to Cauchy-Schwarz inequality, $\alpha_{m,j,k,p}$ is well-defined and
\begin{align*}
\abs{\alpha_{m,j,k,p}} &
\leq\displaystyle{\left(\int_{\gamma_{m,j,k}}{\abs{e_{m,j,k,p}(z)}^2} 
\abs{\frac{\dx z}{\xi_j-z}}\right)^{1/2}}
\displaystyle{\left(\int_{\gamma_{m,j,k}}{\abs{h(z)}^2} 
\abs{\frac{\dx z}{\xi_j-z}}\right)^{1/2}}\\
&\leq \displaystyle{\left(\int_{\gamma_{m,j,k}}{\abs{h(z)}^2} \abs{\frac{\dx z}{\xi_j-z}}\right)^{1/2}}\\
&\leq {\rm log}(\rho)\|h\|_{\infty,E_{s}},
\end{align*}
by (\ref{logrho}). 

Likewise, we have an estimate
$$
\abs{\alpha_{0,j,k,p}}\lesssim \|h\|_{\infty,E_{s}}
$$
for $j \in \{1,\ldots,N\}$, $k \geq 0$ and $p \geq 0$.
This proves the boundedness of the family $\{\alpha_{m,j,k,p}\}$.
By Proposition \ref{CV}, we therefore have
$$
\sum_{m,j,k,p}\vert \alpha_{m,j,k,p}\Phi_{m,j,k,p}(\xi)\vert\,<\infty.
$$

Let $\xi \in E_{r}$. The Cauchy formula, applied to the function 
$\displaystyle{z \mapsto \frac {h(\xi)}{\prod\limits_{j=1}^N(1-\overline{\xi_j}z)^{1/2}}}$,
gives
\begin{align*}
h(\xi)	& = \displaystyle{\frac {1}{2\pi i}\int_{\Gamma} h(z) 
K(z,\xi)\frac{\dx z}{\prod\limits_{j=1}^N(1 -\overline{\xi_{j}}z)}}.
\end{align*}
Then we can decompose $h$ as 
$h=\sum\limits_{m=0}^2\sum\limits_{j=1}^N h_{m,j}$, with
\begin{align*}
h_{m,j}(\xi) &= \displaystyle{\frac {1}{2\pi i}\int_{\Gamma_{m,j}} h(z) 
K(z,\xi)\frac{\dx z}{\prod\limits_{j=1}^N(1 -\overline{\xi_{j}}z)}}\\
&= \displaystyle{\frac {1}{2\pi i}\sum_{k=0}^\infty\int_{\gamma_{m,j,k}} h(z) 
K(z,\xi)\frac{\dx z}{\prod\limits_{j=1}^N(1 -\overline{\xi_{j}}z)}}.
\end{align*}

We decompose the restriction of $h$ to $\gamma_{m,j,k}$ in the Hilbertian basis
$\{e_{m,j,k,p}\}_{p=0}^\infty$ of $H_{m,j,k}$,
$$
\displaystyle{h|_{\gamma_{m,j,k}} = \sum\limits_{p=0}^\infty \alpha_{m,j,k,p}e_{m,j,k,p}}.
$$
Since the series $\displaystyle{\sum\limits_{p\geq 0} \alpha_{m,j,k,p} e_{m,j,k,p}}$ 
converges in $L^2$ to $h|_{\gamma_{m,j,k}}$, then the convergence holds 
a fortiori in $L^1$. We derive that, for all $\xi \in E_{r}$,
\begin{align*}
h(\xi) 
&= \displaystyle{\sum\limits_{m=0}^{2}\sum\limits_{j=1}^N\sum\limits_{k=0}^\infty 
\frac{1}{2\pi i}\int_{\gamma_{m,j,k}}h(z)K(z,\xi) 
\frac{\dx z}{\prod\limits_{j=1}^N(\overline{\xi_{j}}z)}}\\
&= \displaystyle{\frac{1}{2\pi i} \sum\limits_{m,j,k,p} \alpha_{m,j_0,k,p}\int_{\gamma_{m,j,k}}e_{m,j_0,k,p}(z)K(z,\xi) \frac{\dx{z}}{\prod\limits_{j=1}^N(\overline{\xi_{j}}z)}}\\
&= \displaystyle{\sum\limits_{m,j,k,p} \alpha_{m,j,k,p} \Phi_{m,j,k,p}}.
\end{align*}
\end{proof}

In order to define a new quadratic calculus 
in the next section, we need a vectorial version of the decomposition obtained before. For any Banach space $Z$, we let $H^\infty(E_s)$ be the 
Banach space of all bounded holomorphic functions
$ : E_s \rightarrow Z$, equipped with the supremum norm 
$$
\norm{f}_{\infty,E_s}=\sup\bigl\{\norm{f(z)}_Z\, :\ z\in E_s\bigr\}.
$$

Let $h\in H^\infty(E_s;Z)$. As in Proposition \ref{Alpha}, we define
$$
\alpha_{m,j,k,p} = \int_{\gamma_{m,j,k}}
h(z){e_{m,j,k,p}(z)} \abs{\frac{\dx{z}}{\xi_{j}-z}}.
$$
for any $m\in\{1,2\}$, $j\in\{1,\ldots,N\}$,$k\geq 0$ and $p\geq 0$, 
and we define
$$
\alpha_{0,j,k,p} = \int_{\gamma_{0,j,k}}
h(z){e_{m,j,k,p}(z)} \abs{\dx{z}}.
$$
for any $j\in\{1,\ldots,N\}$,$k\in\{0,\ldots,M_j\}$ and $p\geq 0$.
Note that the $\alpha_{0,j,k,p}$ are well-defined
elements of $Z$. Arguing as in the
proof of Proposition \ref{Alpha}, we obtain 
that there exists a constant $C>0$ (not depending on $h$) such that
for all $m \in \{0,1,2\}$, $j = \{1,\ldots,N\}$, $k\geq 0$ and $p\geq 0$,
\begin{equation}\label{estimateAlphaVect}
\norm{\alpha_{m,j,k,p}} \lesssim \norm{h}_{\infty,E_{s}}
\end{equation}

Let $x^*\in Z^*$. By continuity of $x^*$, we have
$$
\langle{x^*,\alpha_{m,j,k,p}}\rangle = 
\int_{\gamma_{m,j,k}}\langle x^*, h(z) 
\rangle {e_{m,j,k,p}(z)} \abs{\frac{\dx{z}}{\xi_{j}-z}}
$$
If we apply Proposition \ref{Alpha} to the function
$\langle{x^*,h(.)}\rangle : E_s \mapsto \C$, we obtain that
$$
\langle x^*, h(z) \rangle = \sum\limits_{m,j,k,p}^\infty \langle{x^*,\alpha_{m,j,k,p}}\rangle \Phi_{m,j,k,p}(z)
$$
for all $z \in E_r$.
By the uniform estimate (\ref{estimateAlphaVect}) and Proposition
\ref{CV}, we can define for all $z \in E_r$, 
$$
\displaystyle{u(z) = \sum\limits_{m,j,k,p}^\infty \alpha_{m,j,k,p} \Phi_{m,j,k,p}(z)}
$$ 
Then, by continuity of $x^*$,
$$
\langle{x^*,u(z)}\rangle = \sum\limits_{m,j,k,p}^\infty \langle{x^*,\alpha_{m,j,k,p}}\rangle \Phi_{m,j,k,p}(z)= \langle{x^*,h(z)}\rangle
$$

This equality is valid for all $x^* \in Z^*$. Hence by the Hahn-Banach theorem, $u(z) = h(z)$ for all $z \in E_r$.

After proceeding with a re-indexation of $\{m,j,k,p\}$, we obtain
the following.

        %Then, for all $z \in E_r$ : 
        %\begin{align*}
        %    \langle x^*, h(z) \rangle &= \langle x^*, \sum_{i=1}^\infty \alpha_i \phi_i(z) \rangle\\
        %    &= \sum_{i=1}^\infty \langle x^*, \alpha_i \rangle \phi(z)\\
        %\end{align*}
        
        %Since $h$ is holomorphic, then $\langle{h,x^* \rangle}$ is holomorphic as well. Since $\abs{\langle{h,x^* \rangle}} \leq \norm{x^*} \norm{h}_{\infty,E_s}$, the function $\langle{h,x^* \rangle}$ is in $\HI_0(E_s)$.
        %If the sequence $(\phi_i)_{i\geq1}$ satisfies the conclusion of this proposition, that is, for a certain $(\alpha_i)_{i\geq1} \in l^\infty(Z)$ : 
        %\begin{equation*}
        %    h = \sum\limits_{i\geq 1} \alpha_i \phi_i
        %\end{equation*}
        %then by continuity of $x^*$, we have : 
        %\begin{equation*}
        %    x^*\circ h = \sum\limits_{i\geq 1} \langle{x^*,\alpha_i\rangle} \phi_i
        %\end{equation*}
        
\begin{proposition}\label{vectorial}
Let $r \in (0,1)$, let $s \in (r,1)$ and let 
$Z$ be a Banach space. There exists a constant $C>0$ and a 
sequence $(\Phi_{i})_{i\geq 1}$ of $\HI_0(E_r)$ 
such that 
$$
\sup\limits_{z \in E_r} \sum\limits_{i=1}^\infty \abs{\Phi_i(z)}^{1/2} < \infty,
$$
and for any $h \in H^\infty(E_s,Z)$, there exists a bounded
sequence $(\alpha_i)_{i\geq1}$ of $Z$ 
such that 
$$
h(z) = \sum\limits_{i=1}^\infty \alpha_i \Phi_i(z)
\quad\hbox{for all}\ z\in E_r  \hspace{1cm} \hbox{and}\hspace{1cm}
\norm{\alpha_i} \leq C \norm{h}_{\infty,E_s}  \quad\hbox{for all}\ i\geq 0.
$$
\end{proposition}

We look now for a factorization of these functions $\phi_{i}$. 
	%LEMMA FACTORIZATION
	\iffalse
	\begin{lemma}\label{factorization}
        There exists two sequences of functions $(psi_{m,j,k,p})$ and $(\tilde{\psi}_{m,j,k,p})$ in $H_0^\infty(E_r)$ such that for all $\xi \in E_r$ and $q$ as in () : 
        \begin{equation}
            \phi_{m,j,k,p}(\xi) = {\psi}_{m,j,k,p}(\xi)\tilde{\psi}_{m,j,k,p}(\xi)
        \end{equation}
        and : 
    \end{lemma}
    \fi
    
    \begin{lemma}\label{factorization}
        There exists two sequences $(\psi_{i})_{i\geq1}$ and $(\tilde{\psi}_i)_{i\geq1}$ of $H_0^\infty(E_r)$ such that for all $\xi \in E_r$ and $q$ as in (\ref{estimateq})
        \begin{equation}
            \Phi_{i}(\xi) = {\psi}_{i}(\xi) \tilde{\psi}_{i}(\xi)
        \end{equation}
        and
        \begin{equation*}
        \sup\limits_{z \in E_r}\sum\limits_{i = 1}^\infty\abs{\psi_{i}(z)} < \infty \qquad 
            \sup\limits_{z \in E_r}\sum\limits_{i = 1}^\infty\abs{\tilde{\psi}_{i}(z)} < \infty
    \end{equation*}
    \end{lemma}
	%PROOF LEMMA FACTORIZATION
	\begin{proof}
	    Since the domain $E_r$ is a simply connected domain bounded by a Jordan rectifiable curve, any function of $H^\infty(E_r)$ admits boundary values and one can apply the inner-outer factorization to all $\Phi_{i}$, which gives two sequences of holomorphic functions $(\psi_{i})_{i\geq 1}$ and $(\tilde{\psi}_{i})_{i\geq 1}$ in $H_0^\infty(E_r)$ such that
	    \begin{equation}
	        \Phi_{i}(\xi) = \psi_{i}(\xi) \tilde{\psi}_{i}(\xi) \text{ for all } \xi \in E_r 
	   \end{equation}
	   and
	   \begin{equation}
	       \abs{\psi_{i}(\xi)} = \abs{\tilde{\psi}_{i}(\xi)} = \abs{\Phi_{i}(\xi)}^{1/2} \text{ for all } \xi \in \partial{E_r}
	    \end{equation}
	\end{proof}
	
	Applying the maximum principle on the domain $\overline{E_r}$ to the functions $\Phi_{i}$, $\psi_{i}$ and $\tilde{\psi}_{i}$ (or more specifically to their prolongation on the whole closed unit disc), and using lemma \ref{estimationsPHI}, we have the existence of a constant $C >0$ such that for all $m \in \{0,1,2\}, j \in \{1,..,N\}, k \geq 0$ and $p \geq 1$
    \begin{equation}\label{psiFinite}
        \sum\limits_{i = 1}^\infty\abs{\psi_{i}(z)} \leq C \qquad 
            \sum\limits_{i = 1}^\infty\abs{\tilde{\psi}_{i}(z)} \leq C
    \end{equation}

    Finally, we can write, for all function $h \in H_0^\infty(E_s)$ and all $z \in E_r$ : 
    \begin{equation*}
        h(z) = \sum\limits_{i=1}^\infty \alpha_i \psi_i(z)\tilde{\psi}_i(z)
    \end{equation*}

\section{Quadratic functional calculus}

%We need a stronger notion than $\HI$ functional calculus in order to prove that any $\mathcal{R}$-Ritt$_E$ operator with square function estimates admits this type of functional calculus. We prove in this section that the notion of quadratic functional calculus implies the bounded one (defined in the beginning of the article) and, with hypothesis of finite cotype, the converse is also true.

First, let's recall the definition of a bounded $\HI$ functional calculus for a Ritt$_E$ operator.

\begin{definition}\label{DEhinfty}
    Let $T$ be a Ritt$_E$ operator of type $r\in(0,1)$ and let $s\in(r,1)$
    We say that $T$ admits a bounded $H^\infty(E_s)$ functional calculus if there exists a constant $K\geq 1$
    such that 
    \begin{equation}\label{HFC}
        \|\phi(T)\| \leq K \|\phi\|_{\infty,E_s},
        \qquad\phi \in H_0^\infty(E_s).
    \end{equation}
\end{definition}

We need a stronger notion in order to connect the latter to square function estimates. 
This is the following, which is a generalization of quadratic functional calculus for Ritt operators (see \cite[Section 6]{LM1}).

\begin{definition}\label{HIquadratic}
Let $s\in(0,1)$ and let $T$ be a Ritt$_E$ operator of type $<s$ on $X$. We say 
that $T$ admits a quadratic $H^\infty(E_s)$ functional calculus if there 
exists a constant $K>0$ such that for any $n \geq 1$, for any $\phi_1\ldots, 
\phi_n$ in $H_0^\infty(E_s)$ and for any $x \in X$ :
\begin{equation}\label{quad}
\bignorm{\sum\limits_{l=1}^n \varepsilon_l \otimes \phi_l(T)x}_{\Rad{X}} \leq K\norm{x} \Bignorm{\left(\sum\limits_{l=1}^n \abs{\phi_l}^2\right)^{1/2}}_{\infty,E_s}
\end{equation}
\end{definition}

\begin{remark}   
If $T$ is a Ritt$_E$ operator with a quadratic $H^\infty(E_s)$
functional calculus, then it has a bounded $H^\infty(E_s)$ functional calculus 
as defined in Definition \ref{DEhinfty}. This follows by applying (\ref{quad}) with $n=1$. 
A converse holds if $X$ has a 
finite cotype. More specifically, we have the following adaptation of \cite[Theorem 6.3]{LM1}.
\end{remark}

\begin{theorem}\label{BDDquadratic}
Assume that $X$ has finite cotype and let $T$ be a Ritt$_E$ operator 
on $X$ with a bounded $\HI(E_r)$ functional calculus. Then $T$ admits 
a quadratic $\HI(E_s)$ functional calculus for any $s \in (r,1)$.
\end{theorem}

\begin{proof}
We use here the adaptation of the Franks-McIntosh decomposition presented in Section \ref{FMI}.
By assumption, $T$ admits a bounded $\HI(E_r)$ functional calculus, hence has type $<r$.
Thus $r$ is $E$-large enough. Let $s \in (r,1)$.
Then by Proposition \ref{vectorial} and Lemma \ref{factorization}, 
there exist a constant $C>0$ and two sequences $(\psi_i)_{i\geq1}$ and 
$(\tilde{\psi}_i)_{i\geq1}$ in $H_0^\infty(E_r)$ such that the following properties 
hold.
        
\smallskip
- For any $z \in E_r$ : 
        \begin{equation}\label{bddsumpsi}
            \sum\limits_{i\geq 1}\abs{\psi_i(z)} \leq C \qquad 
            \hbox{and}\qquad \sum\limits_{i\geq 1}\abs{\tilde{\psi}_i(z)} \leq C.
        \end{equation}

        \smallskip
- For any Banach space $Z$ and any function $\Phi \in H_0^\infty(E_s;Z)$, there exists a bounded sequence $(\theta_i)_{i\geq 1}$ of $Z$
such that :
        \begin{equation*}
            \norm{\theta_i} \leq C \norm{\Phi}_{\HI(E_s;Z)} , \qquad i\geq 1,
        \end{equation*}
        and
        \begin{equation}\label{bi}
            \Phi(z) = \sum\limits_{i=1}^\infty \theta_i \psi_i(z)\tilde{\psi}_i(z) ,  \qquad z \in E_r.
        \end{equation}
        
We assume that the operator $T$ admits a bounded $\HI(E_r)$ functional calculus. This implies that for any sequence $(\eta_i)_{i\geq1}$ of $\{-1,1\}$ and any $m\geq 1$, we have 
        %$$
            %\bignorm{\sum\limits_i \eta_i \psi_i(T)} \lesssim \sup\limits_{z \in E_r} \abs{\sum\limits_{i} \eta_i \psi_i(z)} \leq \sup\limits_{i} \abs{\eta_i} \sup\limits_{z \in E_r} \sum\limits_{i} \abs{\psi_i(z)} 
        %$$
        $$
            \Bignorm{\sum\limits_{i=1}^m \eta_i \psi_i(T)} \lesssim \sup\limits_{z \in E_r} \abs{\sum\limits_{i=1}^m
            \eta_i \psi_i(z)}\,\leq \sup\limits_{z \in E_r} \sum\limits_{i = 1}^m \abs{\psi_i(z)}.
        $$
        %for any sequences $(\eta_i)_{i\geq1}$ of complex numbers. 
        We derive by (\ref{bddsumpsi}) that  
       $$ 
            \sup\limits_{m\geq 1} \sup\limits_{\eta_i = \pm 1} \Bignorm{\sum\limits_{i = 1}^m \eta_i \psi_i(T)} < \infty.
        $$
In turn, this implies  an estimate
\begin{equation}\label{supsup}
\Bignorm{\sum_{i=1}^m 
\varepsilon_i\otimes \psi_i(T)x}_{{\rm Rad}(X)}
\lesssim\norm{x},\qquad x\in X,\, m\geq 1.
\end{equation}
The same arguments yield
\begin{equation}\label{supsup2}
\Bignorm{\sum_{i=1}^m \varepsilon_i\otimes \tilde{\psi_i}
(T)x}_{{\rm Rad}(X)}
\lesssim\norm{x},\qquad x\in X,\, m\geq 1.
\end{equation}

We apply the decomposition property stated above
with $Z = l_n^2$, for some $n \geq 1$. Let $\phi_1,...,\phi_n$ be 
elements of $\HI_0(E_s)$ and consider the element $\Phi \in \HI(E_s;l_n^2)$
defined for all $z \in E_s$ by 
$$
\Phi(z) = (\phi_1(z), ... ,\phi_n(z))
$$
If we write $\theta_i = (\alpha_{i,1}, \alpha_{i,2}, ... , \alpha_{i,n})$ for any $i\geq 1$, where the sequence $(\theta_i)_{i\geq 1}$ is the bounded sequence of $Z = l_n^2$ given by (\ref{bi}), then for any $z \in E_r$ and $l \in \{1,..,n\}$, we have
$$
\phi_{l}(z) = \sum\limits_{i=1}^\infty \alpha_{i,l} \psi_i(z) \tilde{\psi_i}(z)
$$
and 
\begin{equation}\label{supalpha}
\sup\limits_{i} \left( \sum\limits_{l} \abs{\alpha_{i,l}}^2 
\right)^{1/2} \leq C \biggnorm{\left(\sum\limits_{l=1}^n \abs{\phi_l}^2\right)^{1/2}}_{\infty,E_s}.
\end{equation}
For any $l \in \{1,..,n\}$ and any integer $m \geq 1$, we consider the function
$$
h_{m,l} = \sum\limits_{i=1}^m \alpha_{i,l} \psi_i \tilde{\psi_i}.
$$
Mimicking the argument in the proof of \cite[Theorem 6.3]{LM1} and using  (\ref{supsup}), we obtain an estimate
        $$
        \Bignorm{\sum\limits_l \eps_l \otimes h_{m,l}(T)x}_{\Rad{X}} \lesssim \Bignorm{\sum\limits_{i,l}\alpha_{i,l} \,\eps_i \otimes \eps_l \otimes \tilde{\psi}_i(T)x}_{\rara{X}},\quad x\in X.
        $$
        We assumed that $X$ has finite cotype. According to Lemma \ref{Kaiser}, we then have an estimate
        $$
        \Bignorm{\sum\limits_{i,l}\alpha_{i,l} \,\eps_i \otimes \eps_l \otimes \tilde{\psi}_i(T)x}_{\rara{X}}\lesssim
        \sup\limits_{i} \left( \sum\limits_{l} \abs{\alpha_{i,l}}^2 \right)^{1/2}
        \Bignorm{\sum\limits_i \eps_i \otimes \tilde{\psi}_i(T)x}_{\Rad{X}},
        \quad x\in X.
$$
Applying (\ref{supsup2}) and (\ref{supalpha}), this implies
$$
\Bignorm{\sum\limits_{i,l}\alpha_{i,l} \,\eps_i \otimes \eps_l \otimes \tilde{\psi}_i(T)x}_{\rara{X}}\lesssim
\norm{x} \biggnorm{\left(\sum_{l=1}^n \abs{\phi_l}^2 \right)^{1/2}}_{\infty, E_s},
\quad x\in X.
$$
We finally obtain 
        $$
        \Bignorm{\sum\limits_l \eps_l \otimes h_{m,l}(T)x}_{\Rad{X}} \lesssim \norm{x} \biggnorm{\left( \sum_{l=1}^n \abs{\phi_l}^2 \right)^{1/2}}_{\infty, E_s'}.
        $$

We now use
an approximation process to get the above estimate
with $\phi_l$ instead of $h_{m,l}$. 
Consider $\rho \in (0,1)$, 
and let $\Gamma_\rho$ be any contour 
of $\sigma(T_\rho)$ included in $E_r$. Then we can write 
$$
h_{m,l}(\rho T) =\,\frac 1 {2\pi i} \int_{\Gamma_\rho} h_{m,l}(z) R(z,\rho T) \dx z
$$
and
$$
\phi_{l}(\rho T) =\,\frac 1 {2\pi i} \int_{\Gamma_\rho} \phi_{l}(z) R(z, \rho T) \dx z.
$$
Since $(h_{m,l})_{m\geq1}$ tends to $\phi_l$ pointwise on $E_r$
and $\sup\limits_{m,z\in E_r} \abs{h_{m,l}(z)} < \infty$, then
we deduce (by Lebesgue's dominated convergence theorem) that
$$
\lim\limits_{m \rightarrow \infty} h_{m,l}(\rho T) =\phi_l(\rho T).
$$
Furthermore we have 
$$
\lim\limits_{\rho \rightarrow 1} \phi_l(\rho T) = \phi_l(T),
$$
by Lemma \ref{phirhoT}. This yields the expected result.
\end{proof}
        
        %For any $\rho \in (0,1)$, the operator $\rho T$ is bounded, and $\sigma(\rho T)$ is a compact subset of $E_r$. We derive that for any $h \in H_0^\infty(E_r)$, by the Dunford-Riesz functional calculus : 
        
        %\begin{equation}
        %    h(\rho T) = \frac 1 {2\pi i} \int_{\partial E_r} h(z) R(z,\rho T) \dx{z}
        %\end{equation}
        
        %\begin{equation}
        %    \lim\limits_{m\rightarrow \infty} h_{m,l}(\rho T) = f_l(\rho T), l=1,..,n
        %\end{equation}
        
        %\begin{equation}
        %    \lim\limits_{\rho \rightarrow 1} f_l(\rho T) = f_l(T)
        %\end{equation}

\section{Square functions}

    The aim of this section is to compare the boundedness of the square functions $\norm{\,\cdotp}_{T,\alpha}$
    given by Definition \ref{SF} with the
    boundedness of $H^\infty$ functional 
    calculus from
    Definition \ref{DEhinfty}, for any Ritt$_E$ operator.

We start with a numerical estimate.
    
    \begin{lemma}\label{ck}
        Let $(c_k)_{k\geq0}$ be the sequence of complex numbers such that : 
        \begin{equation}\label{ckeq}
            \displaystyle{\frac 1 {\prod\limits_{j=1}^N (1-\overline{\xi_j}z)^3}} = \sum\limits_{k=0}^\infty c_k z^k,\qquad z\in\D.
        \end{equation}
        Then : 
        $$
            \abs{c_k} \lesssim k^2
        $$
    \end{lemma}
    
    \begin{proof}
        According to the decomposition into simple elements, one can find complex numbers $\alpha_{p,j}$ with $p \in \{1,2,3\}$ and $j \in \{1,..,N\}$ such that, for all $z \in \D$ :
        \begin{equation}
            \displaystyle{\frac 1 {\prod\limits_{j=1}^N (1-\overline{\xi_j}z)^3} = \sum\limits_{j = 1}^N \sum\limits_{p=1}^3  \frac {\alpha_{p,j}}{(1-\overline{\xi_j}z)^p}}
        \end{equation}
        Using the following identities, available for any $z \in \D$, 
        $$
        \displaystyle{\frac{1}{1-\overline{\xi_j}z} = \sum\limits_{k=0}^\infty \overline{\xi_j}^kz^k }, 
        \qquad
            \displaystyle{\frac{1}{(1-\overline{\xi_j}z)^2} = \sum\limits_{k=0}^\infty (k+1)\overline{\xi_j}^kz^k },
            $$
            and
            $$
            \displaystyle{\frac{1}{(1-\overline{\xi_j}z)^3} = \sum\limits_{k=0}^\infty \frac{(k+1)(k+2)}2\overline{\xi_j}^kz^k },
        $$
        we can write 
        \begin{equation}\label{dev}
            \displaystyle{\frac 1 {\prod\limits_{j=1}^N (1-\overline{\xi_j}z)^3} = \sum\limits_{j = 1}^N \sum\limits_{p=1}^3  \sum\limits_{k=0}^\infty {\alpha_{p,j}} d_{p,k} \overline{\xi_j}^k z^k },
        \end{equation}
        where $d_{1,k} = 1$, $d_{2,k} = k+1$ and $d_{3,k} = \displaystyle{\frac {(k+1)(k+2)} 2}$, for every $k\geq 0$. Note that $d_{p,k} \lesssim k^2$, for all $p \in \{1,2,3\}$.

 We re-write (\ref{dev}) as
        \begin{equation}
            \displaystyle{\frac 1 {\prod\limits_{j=1}^N (1-\overline{\xi_j}z)^3} = \sum\limits_{k = 0}^\infty \sum\limits_{j=1}^N  \sum\limits_{p=1}^3 {\alpha_{p,j}} d_{p,k} \overline{\xi_j}^k z^k },
        \end{equation}
        from which we deduce that for all $k \geq 0$, 
        $$
        c_k = \sum\limits_{j=1}^N  \sum\limits_{p=1}^3 {\alpha_{p,j}} d_{p,k} \overline{\xi_j}^k
        $$
        Then the expected result follows from the
        above estimates on the coefficients $d_{p,k}$.
    \end{proof}

\begin{lemma}\label{RBdd}
Let $r \in (0,1)$ and let $T : X \rightarrow X$ be a Ritt$_E$ operator of type $r$.
For any $s \in (r,1)$, there exists a constant $K>0$ such that for any polynomial $\phi$ 
and any $k\geq 1$ : 
\begin{equation*}
k\bignorm{\phi(T)T^{k-1}\prod\limits_{j=1}^N(I-\overline{\xi_j}T)} \leq K \norm{\phi}_{\infty,E_s}.
\end{equation*}
If $T$ is $\mathcal {R}$-Ritt$_E$ of type $r$,
then for $s \in  (r,1)$, there exists $K>0$ such 
that for any polynomial $\phi$, the set
$\bigl\{k\phi(T)T^{k-1}\prod\limits_{j=1}^N(I-\overline{\xi_j}T)\, :\, k\geq 1\bigr\}$ is 
${\mathcal R}-$bounded, with
\begin{equation*}
\mathcal{R}\Bigl(\bigl\{k\phi(T)T^{k-1}\prod\limits_{j=1}^N(I-\overline{\xi_j}T)
\, :\, k\geq 1\bigr\}
\Bigr) \leq K\norm{\phi}_{\infty,E_s}.
\end{equation*}
\end{lemma}
    
\begin{proof}
We focus only on the $\mathcal{R}$-Ritt$_E$ case, the proof 
of the Ritt$_E$ case being similar and simpler. 
Fix any $s'\in(s,1)$. For any $k\geq 1$, the polynomial function
$\lambda\mapsto \phi(\lambda)\lambda^{k-1}\prod\limits_{j=1}^N(I-\overline{\xi_j}\lambda)$ 
is an element of $H^\infty_0(E_{s'})$, hence we may write
\begin{equation*}
k\phi(T)T^{k-1}\prod\limits_{j=1}^N(I-\overline{\xi_j}T) = \frac 1 {2\pi i} \int_{\partial E_s} k \phi(\lambda)\lambda^{k-1}\prod\limits_{j=1}^N(I-\overline{\xi_j}\lambda)R(\lambda,T) \dx{\lambda}.
\end{equation*}
Since $T$ is $\mathcal{R}$-Ritt$_E$ of type $r$, it follows from
\cite[Theorem 8.5.2]{HVVW} that there exists a constant $C>0$ such that 
\begin{align*}
\mathcal{R}\Bigl(\bigl\{k\phi(T)T^{k-1}\prod\limits_{j=1}^N(I-\overline{\xi_j}T)
\, :\, k\geq 1\bigr\}
\Bigr) &\leq C \sup\limits_{k\geq 1}\left\{ k\int_{\partial E_s}  \abs{\phi(\lambda)}\abs{\lambda}^{k-1}\abs{\dx{\lambda}}\right\}\\
&\leq C \norm{\phi}_{\infty,E_s} \sup\limits_{k\geq 1}\left\{ k\int_{\partial E_s}  \abs{\lambda}^{k-1}\abs{\dx{\lambda}}\right\}.\\
\end{align*}
The latter quantity is finite, according to the proof of 
Proposition \ref{Rbddalpha}.
\end{proof}

    \iffalse
    \begin{definition}\label{sq1functions}
        Let $T : X \rightarrow X$ a $\mathcal{R}$-Ritt$_E$ operator and $\alpha > 0$. We set for every $x \in X$ :
        $$
            \norm{x}_{T,\alpha} = \lim\limits_{n\rightarrow \infty}\Bignorm{\sum\limits_{k=1}^n k^{\alpha - 1/2} \varepsilon_k \otimes T^{k-1}\prod\limits_{j=1}^N(I-\overline{\xi_j}T)^\alpha(x)}_{\Rad{X}}
        $$
        %$$
        %    \norm{y}_{T^*,\alpha} = \lim\limits_{n\rightarrow \infty}\Bignorm{\sum\limits_{k=1}^n k^{\alpha - 1/2} \varepsilon_k \otimes (T^*)^{k-1}\prod\limits_{j=1}^N(I-\overline{\xi_j}T^*)^\alpha(y)}_{\Rad(X^*)}
        %$$
        (We may have  $\norm{x}_{T,\alpha} = \infty$) %and $\norm{y}_{T^*,\alpha} = \infty$)
    \end{definition}
    
    \begin{remark}
        If $X = H$ is a Hilbert space and $T \in B(H)$ a power bounded operator, then these square functions are defined by 
        $$
            \norm{x}_{T,\alpha} = \lim\limits_{n\rightarrow \infty}\left(\sum\limits_{k=1}^n k^{2\alpha - 1}\bignorm{T^{k-1}\prod\limits_{j=1}^N(I-\overline{\xi_j}T)^\alpha(x)}^2\right)^{1/2}
        $$

        If $X$ is a Banach lattice with finite cotype and $T \in B(X)$ a power bounded operator, then they are defined by 
        $$
            \norm{x}_{T,\alpha} = \lim\limits_{n\rightarrow \infty}\Bignorm{\left(\sum\limits_{k=1}^n k^{2\alpha - 1}\abs{T^{k-1}\prod\limits_{j=1}^N(I-\overline{\xi_j}T)^\alpha(y)}^2\right)^{1/2}}_X
        $$
        
    \end{remark}
    \fi

\begin{proposition}\label{quadraticSFE}
Let $T : X \rightarrow X$ be a Ritt$_E$ operator of type $r \in (0,1)$, and assume that $T$
admits a quadratic $\HI(E_s)$ functional calculus for some $s \in (r,1)$. Then for any $\alpha>0$, we have an estimate
$$
\norm{x}_{T,\alpha} \lesssim \norm{x},\qquad x\in X.
$$
\end{proposition}    

\begin{proof}
We fix $\alpha>0$.
For any $k\geq 1$, let us  define the function
$\phi_k : \D \rightarrow \C$ by
$$
\phi_k(z) = k^{\alpha - 1/2} z^{k-1}
\prod\limits_{j=1}^N(1-\overline{\xi_j}z)^\alpha,
\qquad z\in\D.
$$
By definition \ref{SF}, we may write
$$
\norm{x}_{T,\alpha}=\lim_n\Bignorm{\sum_{k=1}^n
\varepsilon_k\otimes \phi_k(T)x}_{{\rm Rad}(X)} \qquad x\in X.
$$
Hence it suffices to prove a uniform bound of the sums 
$$
\sum\limits_{k=1}^\infty \abs{\phi_k(z)}^2,
$$ 
for $z \in E_s$.
To obtain such a bound, we adapt an argument from
\cite{A} to the Ritt$_E$ case.

Let $x \in [0,1)$. 
We define on $(0,\infty)$ the differentiable 
function $f : t \mapsto t^{2\alpha-1}x^{t-1}$, 
whose derivative is 
$f' : t \mapsto t^{2\alpha-2}x^{t-1}[2\alpha-1 +t\log(x)]$.

If $0<\alpha\leq\frac12$, this derivative is non-positive 
hence $f$ is non-decreasing on $(0,\infty)$. If $\alpha>\frac12$, $f'$
vanishes at $\cfrac{2\alpha - 1}{-\log(x)}$ and we obtain that 
$f$ is non-decreasing on $\left(0,\cfrac{2\alpha-1}{-\log(x)}\right)$ and non-increasing on $\left(\cfrac{2\alpha-1}{-\log(x)},\infty\right)$.

We may therefore apply a comparison criteria to $f$, which
yields the following estimate (not depending on $x$),
$$
\sum\limits_{k=1}^\infty k^{2\alpha-1}x^{k-1} 
\lesssim \int_0^\infty t^{2\alpha-1}x^{t-1}\dx{t}.
$$
Using the changing variable $u = -\log(x)t$ in this estimate, we obtain another one
\begin{equation}\label{estintegral}
\sum\limits_{k=1}^\infty k^{2\alpha-1}x^{k-1} \lesssim \cfrac{\Gamma(2\alpha)}{x(-\log(x))^{2\alpha}}
\end{equation}
where $\Gamma$ is the usual 
function, 
$y\mapsto \displaystyle{\int_0^\infty u^{y-1}e^{-u} \dx u}$.

Now let $z \in E_s$. By (\ref{estintegral}), we have
\begin{align*}
\sum\limits_{k=1}^\infty \abs{\phi_k(z)}^2 &= \prod\limits_{j=1}^N \abs{1-\overline{\xi_j}z}^{2\alpha} \sum\limits_{k=1}^N k^{2\alpha - 1}(\abs{z}^2)^{k-1}\\
&\lesssim \prod\limits_{j=1}^N \abs{1-\overline{\xi_j}z}^{2\alpha} \cfrac {\Gamma(2\alpha)}{\abs{z}^2 (-\log(\abs{z}^2))^{2\alpha}}\\
&= \prod\limits_{j=1}^N \abs{1-\overline{\xi_j}z}^{2\alpha} \cfrac {\Gamma(2\alpha)}{2^{2\alpha}\abs{z}^2 (-\log(\abs{z}))^{2\alpha}}\\
&\lesssim \prod\limits_{j=1}^N \abs{1-\overline{\xi_j}z}^{2\alpha} \cfrac {\Gamma(2\alpha)}{2^{2\alpha}\abs{z}^2 (1-\abs{z})^{2\alpha}}.
\end{align*}
Fix $\delta >0$ and write $V_{\delta,j} := E_s \cap D(\xi_j,\delta)$ for all $j \in \{1,..,N\}$. We may assume that 
$\delta$ is small enough to ensure that 
$D(\xi_j,\delta)\cap E=\{\xi_j\}$. This implies
the existence
of $\rho \in (0,1)$ such that if
$z \in E_s \setminus \bigcup\limits_{j=1}^N V_j$, then
$\abs{z} \leq \rho$. Consequently,
$$
\sum\limits_{k=1}^\infty \abs{\phi_k(z)}^2 \lesssim \sum\limits_{k=1}^\infty k^{2\alpha-1}\rho^{2(k-1)} 
$$
for $z \in E_s \setminus \bigcup\limits_{j=1}^N V_j$. Of course, this
upper bound is finite.

Now consider $z \in V_{j_0}$ for some 
$j_0 \in \{1,\ldots,N\}$. Then 
$$
\prod\limits_{j=1}^N \abs{1-\overline{\xi_j}z}^{2\alpha} = \abs{1-\overline{\xi_{j_0}}z}^{2\alpha} \underset{j\neq j_0}{\prod\limits_{j=1}^N} \abs{1-\overline{\xi_j}z}^{2\alpha} \leq \abs{1-\overline{\xi_{j_0}}z}^{2\alpha} 2^{2\alpha N}.
$$
We derive that for any such $z$,
$$
\sum\limits_{k=1}^\infty \abs{\phi_k(z)}^2 \lesssim \left( \cfrac {\abs{1-\overline{\xi_{j_0}}z}} {1-\abs{z}} \right)^{2\alpha}.
$$
The ratio $\cfrac{\vert 1-z\vert}{1-\vert z\vert}$ is bounded
on any Stolz domain. Hence there
exists a constant $C_{j_0}>0$ such that, 
for any $z \in V_{j_0}$, 
$$
\cfrac{\abs{1-\overline{\xi_{j_0}}z}}{1-\abs{z}} \leq C_{j_0}.
$$
Hence, the sum $\sum\limits_{k=1}^\infty \abs{\phi_k(z)}^2$ is uniformly bounded on $V_{j_0}$ (for all $j_0$).
Altogether we have proved that $\sum\limits_{k=1}^\infty \abs{\phi_k(z)}^2$ is uniformly bounded on
$E_s$, which completes the proof.
\end{proof}

\begin{corollary}\label{SFEalpha}
Let $T : X \rightarrow X$ be a Ritt$_E$ operator 
which admits an $\HI(E_r)$ bounded functional calculus for some $r \in (0,1)$. Assume that $X$ has finite cotype. 
Then for any $\alpha >0$, we have an estimate
$$
\norm{x}_{T,\alpha} \lesssim \norm{x},\qquad x\in X.
$$
\end{corollary}

\begin{proof}
Since $X$ has finite cotype, Theorem \ref{BDDquadratic} ensures that $T$ admits an $\HI(E_s)$ quadratic functional calculus for all $s \in (r,1)$. The result therefore follows
from Proposition \ref{quadraticSFE}.
\end{proof}
        
We will now focus on a partial converse of the above
corollary, in the case when square functions are
defined with the
parameter $\alpha=1$. Similar statements
with other square functions (with arbitrary
$\alpha$) will be considered in the next section.
    
\begin{theorem}\label{SQest}
Let $T$ be an $\mathcal{R}$-Ritt$_E$ operator of $\mathcal{R}$-type $r \in (0,1)$. Assume that both $T$ and $T^*$ satisfy estimates
$$
\norm{x}_{T,1} \lesssim \norm{x},\qquad x\in X,
$$
and
$$
\norm{y}_{T^*,1} \lesssim \norm{y},\qquad y\in X^*.
$$
Then for all $s \in (r,1)$, $T$ admits a bounded $\HI(E_s)$ calculus.
\end{theorem}

\begin{proof}
%Let $\phi \in \HI_0(E_s)$ be a polynomial. It necessarily vanishes at $\xi_j$, for any $j \in \{1,..,N\}$. 
We consider the sequence $(c_k)_{k\geq0}$ defined in 
Lemma \ref{ck}. For any $\rho \in (0,1)$, the
series $\sum\limits_{k\geq0}c_k(\rho T)^k$ is 
absolutely convergent in $B(X)$, and applying the identity 
from Lemma \ref{ck}, we have
\begin{equation}\label{ckrho}
\sum\limits_{k=0}^\infty c_k(\rho T)^k 
\prod\limits_{j=1}^N(I-\overline{\xi_j}(\rho T))^3 = I_X
\end{equation}
Consider momentarily some
$x \in \ran\Bigl(\prod\limits_{j=1}^N(I-\overline{\xi_j}T)\Bigr)$. Thus, there exists $a \in X$ such that 
$x = \prod\limits_{j=1}^N(I-\overline{\xi_j}T)a $. According to Proposition \ref{Rbddalpha} (with $\alpha=4$) and using the estimate $\abs{c_k} \lesssim k^2$, we have
$$
\bignorm{c_k T^k \prod\limits_{j=1}^N(I-\overline{\xi_j}T)^3x} =  
\abs{c_k} \bignorm{T^k \prod\limits_{j=1}^N(I-\overline{\xi_j}T)^4a}
\lesssim \frac 1 {k^2} \norm{x}.
$$
We deduce that the series $\displaystyle{\sum
\limits_{k=0}^\infty c_k T^k \prod\limits_{j=1}^N(I-\overline{\xi_j}T)^3x}$
is absolutely convergent in $X$. Further,
using again Proposition \ref{Rbddalpha}, we actually have
$$
\sum
\limits_{k=0}^\infty \abs{c_k}
\sup_{\rho\in(0,1]}\Bignorm{(\rho T)^k \prod\limits_{j=1}^N(I-\overline{\xi_j}(\rho T))^3x}
\,<\infty.
$$
We therefore deduce from (\ref{ckrho}) that
\iffalse

For $\rho \in (0,1)$, the series $\displaystyle{\sum\limits_{k=0}^\infty c_k (\rho T)^k \prod\limits_{j=1}^N(I-\overline{\xi_j}(\rho T))^3x}$ is absolutely convergent in $X$ as well.\\
        We can define the application $u_x : (0,1] \rightarrow X$ by :
        $$
            u_x : \rho \mapsto \sum\limits_{k=0}^\infty c_k (\rho T)^k \prod\limits_{j=1}^N(I-\overline{\xi_j}(\rho T))^3x
        $$
        This application is well defined, by a normally convergent series on $(0,1]$, so :
        $$
            u_x(1) = \lim\limits_{\rho \rightarrow 1^{-}} u_x(\rho)
        $$
        According to (\ref{ckrho}), we have $u_x(\rho) = x$, for all $\rho \in (0,1)$. Then $u_x(1)=x$, that is :
        \fi
\begin{equation}\label{idCK}
\sum\limits_{k=0}^\infty c_k T^k \prod\limits_{j=1}^N(I-\overline{\xi_j}T)^3x = x.
\end{equation}

We will use the classical fact that for 
any finite families $(x_k)_{k= 1,..,n}$ of $X$ and $(y_k)_{k=1,..,n}$ of $X^{*}$, we have
\begin{equation*}
\sum\limits_{k=1}^n \langle{x_k,y_k}\rangle = \displaystyle{\int_\Omega \left\langle{\sum\limits_{k=1}^n\varepsilon_k(u)x_k,\sum\limits_{k = 1}^n\varepsilon_k(u)y_k}\right\rangle\dx{\P(u)}},
\end{equation*}
and hence, by the Cauchy-Schwarz inequality,
\begin{equation}\label{CSI}
\abs{\sum\limits_{k=1}^n \langle{x_k,y_k}\rangle} \leq \Bignorm{\sum\limits_{k=1}^n\varepsilon_k(u)x_k}_{\Rad{X}}\Bignorm{\sum\limits_{k=1}^n\varepsilon_k(u)y_k}_{\Rad(X^{*})}.
\end{equation}

For any $n \geq 1$, we have
\begin{align*}
\sum\limits_{k=0}^{3n+2} c_k T^k 
\prod\limits_{j=1}^N(I-\overline{\xi_j}T)^3 
%&= \sum\limits_{k=0}^{n} c_{3k} T^{3k} \prod\limits_{j=1}^N (I-\overline{\xi_j}T)^3 \\
%&+ \sum\limits_{k=0}^{n} c_{3k+1} T^{3k+1} \prod\limits_{j=1}^N (I-\overline{\xi_j}T)^3 \\
%&+ \sum\limits_{k=0}^{n} c_{3k+2} T^{3k+2} \prod\limits_{j=1}^N (I-\overline{\xi_j}T)^3 \\
&= \sum\limits_{k=0}^{n} c_{3k} T^{3k} \prod\limits_{j=1}^N (I-\overline{\xi_j}T)^3 \\
&+ T \sum\limits_{k=0}^{n} c_{3k+1} T^{3k} \prod\limits_{j=1}^N (I-\overline{\xi_j}T)^3 \\
&+ T^2 \sum\limits_{k=0}^{n} c_{3k+2} T^{3k} \prod\limits_{j=1}^N (I-\overline{\xi_j}T)^3.
\end{align*}

We let $\phi$ be a polynomial 
belonging to $H^\infty_0(E_s)$.
We introduce, for all $n\geq 1$, the operators 
$$
V_{n,0} = \sum\limits_{k=0}^n c_{3k} \phi(T) T^{3k}\prod\limits_{j=1}^N (I - \overline{\xi_j}T)^3,
$$
$$
V_{n,1} = T\sum\limits_{k=0}^n c_{3k+1} \phi(T) T^{3k}\prod\limits_{j=1}^N (I - \overline{\xi_j}T)^3,
$$
$$
V_{n,2} = T^2 \sum\limits_{k=0}^n c_{3k+2} \phi(T) T^{3k}\prod\limits_{j=1}^N (I - \overline{\xi_j}T)^3,
$$
so that
$$
\sum\limits_{k=0}^\infty c_k \phi(T)
T^k \prod\limits_{j=1}^N(I-\overline{\xi_j}T)^3
=V_{n,0}+V_{n,1}+V_{n,2}.
$$
The assumption on $\phi$ ensures that 
$\prod\limits_{j=1}^N (1-\overline{\xi_j}z)$ is a 
factor of $\phi$, that is, $\phi(z)=\prod\limits_{j=1}^N (1-\overline{\xi_j}z)Q(z)$
for some polynomial $Q$. Then $\phi(T)=\prod\limits_{j=1}^N (1-\overline{\xi_j}T)Q(T)$.
Hence 
for any $x \in X$, 
$\phi(T)x \in \ran\Bigl(\prod\limits_{j=1}^N (1-\overline{\xi_j}T)\Bigr)$. Consequently,
(\ref{idCK}) is valid with $\phi(T)x$ instead of 
$x$, that is,
\begin{equation}\label{Valid}
\sum\limits_{k=0}^\infty c_k \phi(T)
T^k \prod\limits_{j=1}^N(I-\overline{\xi_j}T)^3x = 
\phi(T)x,\qquad x\in X.
\end{equation}

Let $x \in X$ and $y \in X^*$. We have
\begin{align*}
\abs{\langle{V_{n,0} x,y\rangle} }
&= \abs{\langle{\sum
\limits_{k=0}^n c_{3k} \phi(T) T^{3k}
\prod\limits_{j=1}^N 
(I - \overline{\xi_j}T)^3 x,y\rangle} }\\
&=
\abs{\sum\limits_{k = 0}^n{c_{3k}\langle{[T^k\phi(T)\prod\limits_{j=1}^N(1-\overline{\xi_j}T)]T^k\prod\limits_{j=1}^N(1-\overline{\xi_j}T)x, {{T^*}^k}\prod\limits_{j=1}^N(1-\overline{\xi_j}{T}^{*})y}\rangle}}\\
&= 
\abs{\sum\limits_{k = 0}^n{\frac{c_{3k}}{k^2}\langle{[kT^k\phi(T)\prod\limits_{j=1}^N(1-\overline{\xi_j}T)]k^{1/2}T^k\prod\limits_{j=1}^N(1-\overline{\xi_j}T)x, k^{1/2}{{T^{*}}^k}\prod\limits_{j=1}^N(1-\overline{\xi_j}T^{*})y}\rangle}}
\end{align*}
Hence by Lemma \ref{RBdd} and by
(\ref{CSI}) applied with
$$ 
x_k = \frac{c_{3k}}{k^{1/2}}[T^k\phi( T)\prod\limits_{j=1}^N(1-\overline{\xi_j}T)]T^k\prod\limits_{j=1}^N(1-\overline{\xi_j}T)x
\quad\hbox{and}\quad
y_k = k^{1/2}{T^*}^k\prod\limits_{j=1}^N(1-\overline{\xi_j}{T}^{*})y,
$$
we have the following inequalities,
\begin{align*}
\abs{\langle{V_{n,0} x,y\rangle} }
&\leq
\Bignorm{\sum\limits_{k = 0}^n \eps_k\otimes \frac{c_{3k}}{k^2} [ k T^k\phi(T)\prod\limits_{j=1}^N(1-\overline{\xi_j}T)]k^{1/2}T^k\prod\limits_{j=1}^N(1-\overline{\xi_j}T)x}_{\Rad{X}}\\
&\times
\Bignorm{\sum\limits_{k = 0}^n \eps_k\otimes k^{1/2}{T^{*}}^{k}\prod\limits_{j=1}^N(1-\overline{\xi_j}T^{*})y}_{\Rad(X^*)}\\
&\lesssim
R\Bigl(\Bigl\{k T^k \phi(T) \prod\limits_{j=1}^N(1-\overline{\xi_j} T) : k \geq 0\Bigr\}\Bigr) \norm{x}_{T,1} \norm{y}_{T^*,1}\\
&\lesssim 
\norm{\phi}_{\infty,E_s} \norm{x}_{T,1} \norm{y}_{T^*,1}.
\end{align*}
In the same way, we get the 
other two estimates
$$
\abs{\langle V_{n,1} x , y \rangle} \lesssim \norm{\phi}_{\infty,E_s} \norm{x}_{T,1} \norm{y}_{T^*,1} \qquad \textrm{and} \qquad \abs{\langle V_{n,2} x , y \rangle} \lesssim \norm{\phi}_{\infty,E_s} \norm{x}_{T,1} \norm{y}_{T^*,1}.
$$
Finally, combining these three estimates, we have
$$
\abs{\langle{\sum\limits_{k=0}^{3n+2} c_k T^k \prod\limits_{j=1}^N (I-\overline{\xi_j}T)^3 \phi(T) x,y}\rangle} \lesssim \norm{\phi}_{\infty,E_s} \norm{x}_{T,1} \norm{y}_{T^*,1}
$$
Hence, letting $n\to \infty$ and using
(\ref{Valid}), we deduce an estimate
$$
\abs{\langle{\phi(T) x,y}\rangle} \lesssim \norm{\phi}_{\infty,E_s} \norm{x}_{T,1} \norm{y}_{T^*,1}
$$
Since the elements $x$ and $y$ 
were arbitrary in $X$ and $X^*$, respectively, we deduce an estimate 
$$
\norm{\phi(T)} \lesssim \norm{\phi}_{\infty,E_s}
$$ 
for any polynomial $\phi$ in $\HI_0(E_s)$. 

Now the argument in the proof \cite[Proposition 3.4]{BLM} 
shows that this estimate implies that $T$ admits a bounded $\HI(E_s)$ functional calculus.
\end{proof}

In the case of the Banach space $X$ has property $(\Delta)$, we have the following, consequence of a similar result for sectorial operators.

\begin{proposition}\label{DeltaRittR}
Assume that $X$ has property $(\Delta)$. 
Let $T$ be a Ritt$_E$ operator on $X$.
If $T$ admits a bounded $\HI(E_s)$ functional calculus for some $s \in (0,1)$, then $T$ is $\mathcal{R}$-Ritt$_E$.
\end{proposition}

\begin{proof}
For any $j \in \{1,..,N\}$, we consider the operator $A_j = I - \overline{\xi_j}T$. Then $A_j$ is sectorial, with a bounded $\HI(\Sigma_{\omega})$ functional calculus, for some $\omega \in (0,\frac \pi 2)$.
        
We deduce by \cite[proposition 7.4]{LM0} that each $A_j$ is $\mathcal{R}$-sectorial of $\mathcal{R}$-type $\omega' \in (0,\frac \pi 2)$.
With the same arguments as in the proof of \cite[Lemma 2.7]{BLM}, it follows that $T$ is $\mathcal{R}$-Ritt$_E$.
\end{proof}

Any Banach space with property $(\Delta)$ has finite cotype. Hence, combining Proposition \ref{DeltaRittR} with Corollary \ref{SFEalpha}, we deduce the following equivalence.

\begin{corollary}
Let $X$ be a Banach space with property $(\Delta)$ and let $T : X \rightarrow X$ be a Ritt$_E$ operator. The following are equivalent.
\begin{enumerate}
\item $T$ admits a bounded $\HI(E_s)$ functional calculus for some $s \in (0,1)$;
\item $T$ is $\mathcal{R}$-Ritt$_E$ and $T$ and $T^*$ both satisfy uniform estimates
$$\norm{x}_{T,1} \lesssim \norm{x} \quad \textrm{and} \quad \norm{y}_{T^*,1} \lesssim \norm{y}
$$
for $x \in X$ and $y \in X^*$.
\end{enumerate}
\end{corollary} 

\begin{remark}
This corollary is satisfied if $X$ is UMD, since these Banach spaces have property $(\Delta)$.%(see Proposition \ref{UMDisDELTA}).
\end{remark}

\section{Equivalence of square functions}

The main goal of this section is to prove that if an operator $T: X \to X$ is Ritt$_E$
on a reflexive Banach space with finite cotype, then square functions associated are all equivalent, as presented in theorem \ref{eqSF}.

The following is elementary.

\begin{lemma}\label{squareFctCommute}
Let $T \in B(X)$ be a Ritt$_E$ operator and let $\alpha >0$. 
Let $S \in B(X)$ such that $ST = TS$. Then, for all $x \in X$, we have
$$
\norm{S(x)}_{T,\alpha} \leq \norm{S}\norm{x}_{T,\alpha}.
$$
\end{lemma}

\begin{proof}
Let $x \in X$ and let $x_k := k^{\alpha-\frac 1 2}T^{k-1}\prod\limits_{j=1}^N(I-\overline{\xi_j}T)^\alpha x$ for all $k\geq 1$. Given any $n\geq 1$, we have
$$
\Bignorm{\sum\limits_{k=1}^n\varepsilon_k\otimes S(x_k)}_{\Rad{X}} \leq \norm{S} \Bignorm{\sum\limits_{k=1}^n\varepsilon_k\otimes x_k}_{\Rad{X}} \leq \norm{S}\norm{x}_{T,\alpha}.
$$
Since $ST = TS$, we have 
$$
S(x_k) = k^{\alpha-\frac 1 2}\left[ST^{k-1}\prod\limits_{j=1}^N(I-\overline{\xi_j}T)^\alpha\right](x) = k^{s-\frac 1 2}\left[T^{k-1}\prod\limits_{j=1}^N(I-\overline{\xi_j}T)^\alpha\right](S(x))
$$
for all $k\geq 1$. The result follows at once.
\end{proof}

The following is the main result of this section.
    
\begin{theorem}\label{eqSF}
Assume that $X$ is reflexive and has finite cotype.
Let $T : X \rightarrow X$ be a $\mathcal{R}$-Ritt$_E$ operator. 
Then, for any $\alpha,\beta >0$, we have an equivalence 
$$
\norm{x}_{T,\alpha} \approx \norm{x}_{T,\beta},\qquad x\in X.
$$
\end{theorem}

\begin{proof}
We only show that $\norm{x}_{T,\alpha} \lesssim \norm{x}_{T,\beta}$, the roles of $\alpha$ and $\beta$ being symmetric.
First, we set $M = E(\alpha) + 1$, where $E(\alpha)$ is the integral
part of $\alpha$. Then we set $\gamma = M - \alpha$. By construction, $\gamma>0$.

Let $(c_k)_{k\geq 1}$ be the sequence of complex numbers such that
$$
\sum\limits_{k=1}^\infty c_k k^{\alpha-\frac 1 2} z^{k-1} = \frac 1 {\prod\limits_{j=1}^N (1-\overline{\xi_j}z)^M},
\qquad z\in\mathbb D.
$$
Arguing as in the proof of  Lemma \ref{ck}, we obtain that 
$c_k k^{\alpha - \frac 1 2} = O(k^{M-1})$. This implies that 
$$
c_k = O(k^{\gamma - \frac 1 2}).
$$
The operator $T$
is power bounded hence for any $\rho \in (0,1)$, we have
$$
I=\sum\limits_{k=1}^\infty c_k k^{\alpha-\frac 1 2} (\rho T)^{k-1} \prod\limits_{j=1}^N (I - \overline{\xi_j}\rho T)^M.
$$
Using $M=\alpha+\gamma$, we deduce
the following identities,
\begin{align*}
I 
= &\sum\limits_{k=0}^\infty c_{2k+1} (2k+1)^{\alpha-\frac 1 2} (\rho T)^{2k} \prod\limits_{j=1}^N (I - \overline{\xi_j}\rho T)^M \\
&+ \rho T \left[\sum\limits_{k=0}^\infty c_{2k+2} (2k+2)^{\alpha-\frac 1 2} (\rho T)^{2k} \prod\limits_{j=1}^N (I - \overline{\xi_j}\rho T)^M\right]\\
=& \sum\limits_{k=0}^\infty c_{2k+1} (\rho T)^{k} \prod\limits_{j=1}^N (I - \overline{\xi_j}\rho T)^\gamma (2k+1)^{\alpha-\frac 1 2}(\rho T)^k\prod\limits_{j=1}^N (I - \overline{\xi_j}\rho T)^{\alpha} \\
& + \rho T \left[\sum\limits_{k=0}^\infty c_{2k+2}(\rho T)^k \prod\limits_{j=1}^N (I - \overline{\xi_j}\rho T)^\gamma (2k+2)^{\alpha-\frac 1 2} (\rho T)^{k} \prod\limits_{j=1}^N (I - \overline{\xi_j}\rho T)^\alpha \right].
\end{align*}

For any integer $m \geq 1$ and any
$x \in X$, we deduce from above that
\begin{align*}
&m^{\beta-\frac 1 2}(\rho T)^{m-1}\prod\limits_{j=1}^N (I - \overline{\xi_j}\rho T)^\beta x \\
&= \sum\limits_{k=0}^\infty c_{2k+1}m^{\beta-\frac 1 2}(\rho T)^{m+k-1} \prod\limits_{j=1}^N(I-\overline{\xi_j}\rho T)^{\beta+\gamma}(2k+1)^{\alpha-\frac 1 2}(\rho T)^k \prod\limits_{j=1}^N(I-\overline{\xi_j}\rho T)^\alpha x\\
&+ \rho T\sum\limits_{k=0}^\infty c_{2k+2}m^{\beta-\frac 1 2}(\rho T)^{m+k-1} \prod\limits_{j=1}^N(I-\overline{\xi_j}\rho T)^{\beta+\gamma}(2k+2)^{\alpha-\frac 1 2}(\rho T)^k \prod\limits_{j=1}^N(I-\overline{\xi_j}\rho T)^\alpha x.
\end{align*}
We set 
\begin{equation*}
A_{m,k,1}(\rho) = c_{2k+1}m^{\beta - \frac 1 2} (\rho T)^{m+k-1} \prod\limits_{j=1}^N(I -\overline{\xi_j}\rho T)^{\beta +\gamma},
\end{equation*}
\begin{equation*}
 A_{m,k,2}(\rho) = c_{2k+2}m^{\beta - \frac 1 2} (\rho T)^{m+k-1}
 \prod\limits_{j=1}^N(I -\overline{\xi_j}\rho T)^{\beta +\gamma},
\end{equation*}
\begin{equation*}
y_m(\rho) = m^{\beta - \frac 1 2}(\rho T)^{m-1}\prod\limits_{j=1}^N(I-\overline{\xi_j}\rho T)^{\beta}x,
\end{equation*}
\begin{equation*}
z_{k,1}(\rho) = (2k+1)^{\alpha - \frac 1 2} (\rho T)^{k} 
\prod\limits_{j=1}^N(I-\overline{\xi_j}\rho T)^\alpha x
\end{equation*}
\begin{equation*}
z_{k,2}(\rho) = (2k+2)^{\alpha - \frac 1 2} (\rho T)^{k} \prod\limits_{j=1}^N(I-\overline{\xi_j}\rho T)^\alpha x.
\end{equation*}
Thus, we can write :
$$
y_m(\rho) = \sum\limits_{k=0}^\infty
A_{m,k,1}(\rho)z_{k,1}(\rho) + \rho T 
\sum\limits_{k=0}^\infty A_{m,k,2}(\rho)z_{k,2}(\rho).
$$
For any $n\geq 1$, we consider the partial sums : 
$$
y_{m,n}(\rho)= \sum\limits_{k=0}^n 
A_{m,k,1}(\rho)z_{k,1}(\rho) + \rho T 
\sum\limits_{k=0}^n A_{m,k,2}(\rho)z_{k,2}(\rho).
$$
We  write, for convenience,
$$
y_{m,n,1}(\rho)= \sum\limits_{k=0}^n A_{m,k,1}(\rho)z_{k,1}(\rho)
\qquad\hbox{and}\qquad 
y_{m,n,2}(\rho)= \sum\limits_{k=0}^n A_{m,k,2}(\rho)z_{k,2}(\rho).
$$
We have $y_{m,n}(\rho) \rightarrow y_{m}(\rho)$ when $n \rightarrow \infty$.
        
We would like to control the norms of the operators ${\rm Rad}(X)\to
{\rm Rad}(X)$ with operator valued kernels equal to either
$\bigl[A_{m,k,1}(\rho)\bigr]_{m\geq 1,k\geq 0}$
or $\bigl[A_{m,k,2}(\rho)\bigr]_{m\geq 1,k\geq 0}$.
We focus on the $A_{m,k,1}(\rho)$ only, the 
other case being similar. We write for any $m\geq 1 ,k\geq 0$ : 
$$
A_{m,k,1}(\rho) = \frac{m^{\beta - 
\frac 1 2}c_{2k+1}}{(m+k)^{\beta + \gamma}} 
\Bigl[(m+k)^{\beta+\gamma} (\rho T)^{m+k-1}
\prod\limits_{j=1}^N(I-\overline{\xi_j}\rho T)^{\beta+\gamma}
\Bigr]
$$
Since $c_k = O(k^{\gamma - \frac 1 2})$, we have a uniform estimate
$$
\frac{m^{\beta - \frac 1 2}\abs{c_{2k+1}}}{(m+k)^{\beta+\gamma}} 
\lesssim \frac{m^{\beta - \frac 1 2}k^{\gamma-\frac 1 2}}{(m+k)^{\beta + \gamma}}
$$
It therefore follows from \cite[Proposition 2.3 and Lemma 2.4]{ALM}   
that the infinite matrix 
$$
\left[\frac{m^{\beta-\frac 1 2}
\vert c_{2k+1}\vert}{(m+k)^{\beta+\gamma}}\right]_{k\geq 0,m\geq 1}
$$
represents an element of $B(l^2)$. 

By Proposition \ref{Rbddalpha}, the set
$$
F = \{(m+k)^{\beta+\gamma} (\rho T)^{m+k-1}\prod\limits_{j=1}^N(I-\overline{\xi_j}\rho T)^{\beta+\gamma} : m, k \geq 1, \rho \in (0,1]\}
$$
is $\mathcal{R}$-bounded. Since $X$ has finite cotype, the Rademacher 
and Gaussian averages on $X$ are equivalent. 
The set $F$ is therefore $\gamma$-bounded as well. 
For any finite families of Rademacher variables $(\varepsilon_k)_{k\geq 1}$ and standard Gaussian variables $(\gamma_k)_{k\geq 0}$ we have,
by \cite[Proposition 2.6]{ALM}, %\cite[Proposition 2.6]{ALM},
\begin{align*}
\Bignorm{\sum\limits_{m\geq 1} \varepsilon_m \otimes y_{m,n,1}(\rho)}_{\Rad{X}} 
&= \Bignorm{\sum\limits_{m\geq 1}\sum\limits_{k=0}^n
\varepsilon_m \otimes A_{m,k,1}(\rho)z_{k,1}(\rho)}_{\Rad{X}}\\
&\approx \Bignorm{\sum\limits_{m\geq 1}\sum\limits_{k=0}^n 
\gamma_m \otimes A_{m,k}(\rho)z_{k,1}(\rho)}_{\Gauss(X)}\\
&\lesssim \gamma(F) 
\Biggnorm{\left[\frac{m^{\beta-\frac 1 2}\vert 
c_{2k+1}\vert}{(m+k)^{\beta+\gamma}}
\right]}_{B(l^2)} 
\Bignorm{\sum\limits_{k=0}^n\gamma_k 
\otimes z_{k,1}(\rho)}_{\Gauss(X)}\\
&\lesssim \Bignorm{\sum\limits_{k=0}^n\gamma_k \otimes z_{k,1}(\rho)}_{\Gauss(X)}\\
&\lesssim \Bignorm{\sum\limits_{k=0}^n\varepsilon_k \otimes z_{k,1}(\rho)}_{\Rad{X}}.
\end{align*}
Passing to the limit when $n\to\infty$ and applying
the same argument to $y_{m,n,2}(\rho)$, we finally obtain 
an estimate
$$
\norm{x}_{\rho T,\beta} \lesssim \norm{x}_{\rho T,\alpha}
$$
up to a constant not depending on $\rho$.

It follows from above and to Lemma \ref{SQlimrho1} that we have 
a uniform estimate 
\begin{equation}\label{estimateAlphaBetaOnRange}
\norm{x}_{T,\beta} \lesssim \norm{x}_{T,\alpha},
\end{equation}
for $x \in \ran{\left(\prod\limits_{j=1}^N(I-\overline{\xi_j}T)\right)}$.

We set for any integer $m \geq 1$ and any integer $j \in \{1,..,N\}$
$$
\Lambda_{m,j} = \frac 1 {m+1} \sum\limits_{k=0}^m 
\left(I-(\overline{\xi_j}T)^k\right).
$$
If $x \in \ker(I-\overline{\xi_j}T)$, then 
$\Lambda_{m,j}(x) = 0$. Moreover, if  
$x \in \ran(I - \overline{\xi_j}T)$, let 
$y \in X$ such that $x = (I - \overline{\xi_j}T)y$. Then,
$$
\Lambda_{m,j}(x) = x - \cfrac 1 {m+1} [I - (\overline{\xi_j}T)^{m+1}]y.
$$
Since $T$ is power bounded, we deduce 
$$
\lim\limits_{m\rightarrow \infty} \Lambda_{m,j}(x) = x.
$$
Since $T$ is power bounded, the sequence $(\Lambda_{m,j})_m$ is bounded, hence
the above property holds as well
when $x \in \overline{\ran(I - \overline{\xi_j}T)}$.

It follows that if we denote $P_j$ the
projection on $\ker(I-\overline{\xi_j}T)$ of kernel 
$\overline{\ran(I - \overline{\xi_j}T)}$, then for any $x \in X$,
$$
\lim\limits_{m \rightarrow \infty}\Lambda_{m,j}(x) = (I-P_j)(x).
$$

We consider for any $m \geq 0$ the operator 
$$
\Lambda_m = \prod\limits_{j=1}^N \Lambda_{m,j}.
$$
Thus, for any $x \in X$
\begin{equation}\label{limitLambdam}
\lim\limits_{m \rightarrow\infty} \Lambda_m(x) = 
\prod\limits_{j=1}^N (I-P_j)(x).
\end{equation}

Let $x \in X$. It follows from the definition of the $\Lambda_{m,j}$ 
that for any $m\geq 1$, 
$\Lambda_m(x) \in \ran\left(
\prod\limits_{j=1}^N(I-\overline{\xi_j}T)\right)$.
Hence by (\ref{estimateAlphaBetaOnRange}),
$$
\bignorm{\Lambda_m(x)}_{T,\beta} 
\lesssim \bignorm{\Lambda_m(x)}_{T,\alpha}.
$$
The operators $\Lambda_m$ commute with $T$, 
hence by lemma \ref{squareFctCommute},
$$
\bignorm{\Lambda_m(x)}_{T,\beta} 
\leq \bignorm{\Lambda_m}\bignorm{x}_{T,\alpha}.
$$
We deduce that for all integer $n\geq 1$, we have
$$
\Bignorm{\sum\limits_{k=1}^n k^{\beta-\frac 1 2} 
\varepsilon_k \otimes T^{k-1}\prod\limits_{j=1}^N
(I-\overline{\xi_j}T)^{\beta}\Lambda_m(x)}_{\Rad{X}} 
\lesssim \norm{x}_{T,\alpha}.
$$
Passing to the limit when $m \rightarrow \infty$, we have by (\ref{limitLambdam}),
$$
 \Bignorm{\sum\limits_{k=1}^n k^{\beta-\frac 1 2} \varepsilon_k 
 \otimes T^{k-1}\prod\limits_{j=1}^N(I-\overline{\xi_j}T)^{\beta}
 \prod\limits_{j=1}^N(I-P_j)x}_{\Rad{X}} \lesssim \norm{x}_{T,\alpha}.
$$
Since $P_j(x)$ belongs to $\ker(I-\overline{\xi_j}T)$ 
and $\ker(I-\overline{\xi_j}T)^\beta = \ker(I-\overline{\xi_j}T)$, 
we actually obtain 
$$
\Bignorm{\sum\limits_{k=1}^n k^{\beta-\frac 1 2} \varepsilon_k 
\otimes T^{k-1}\prod\limits_{j=1}^N(I-\overline{\xi_j}T)^{\beta}x}_{\Rad{X}} 
\lesssim \norm{x}_{T,\alpha}.
$$
We deduce that $\norm{x}_{T,\beta} \lesssim \norm{x}_{T,\alpha}$. This proves 
the expected result.
\end{proof}

Combining theorems \ref{eqSF} and \ref{SQest}, we deduce a condition of the existence of bounded functional calculus with more general square functions estimates, but in the reflexive case only.
    
\begin{corollary}\label{implyR-Ritt}
Assume that $X$ is reflexive and that both 
$X$ and $X^*$ have finite cotype. 
Let $T : X \rightarrow X$ be a 
$\mathcal{R}$-Ritt$_E$
operator of $\mathcal{R}$-type 
$r\in (0,1)$ and assume that 
$T^*$ is $\mathcal{R}$-Ritt$_E$ as well. 
If $\alpha$ and $\beta$ are two positive real
numbers such that 
$\norm{x}_{T,\alpha} \lesssim \norm{x}$ 
and  $\norm{y}_{T^*,\beta} \lesssim \norm{y}$ 
for any $(x,y) \in X \times X^*$, 
then for any $s \in (r,1)$, $T$ admits a 
bounded $\HI(E_s)$ functional calculus.\end{corollary}

\section{Applications}
In this last section, we give consequences  
of the previous corollaries in particular 
cases, when $X$ has specific 
geometric properties (for instance when 
$X$ is UMD), when $X$ is a Hilbert 
space or a Banach lattice.

In these situations, $X$ has finite cotype, then some hypothesis of the previous statements can be lightened.

    We start with the case where $X$ has property $(\Delta)$. The only useful specificity is that $X$ has a finite cotype. %(see Example \ref{exdelta}). 
    We need to add duality assumptions.

    \begin{corollary}\label{EquivD}
        Let $X$ be a reflexive Banach space and assume that both $X$ and $X^*$
        have property $(\Delta)$. Let $\alpha, \beta >0$ and let $T : X \rightarrow X$
        be a Ritt$_E$ operator on $X$. Then the following assertions are equivalent : 
        \begin{enumerate}
            \item $T$ admits a bounded $\HI(E_s)$ functional calculus for some $s \in (0,1)$
            \item $T$ and $T^*$ are 
            $\mathcal{R}$-Ritt$_E$ and we have estimates, for any 
            $(x,y) \in X \times X^*$, 
            $$
                \norm{x}_{T,\alpha} \lesssim \norm{x} \quad\hbox{and}\qquad \norm{y}_{T^*,\beta} \lesssim \norm{y}
            $$
        \end{enumerate}
    \end{corollary}

    In the $K$-convex case, a stronger statement can be formulated, using properties recalled in chapter 1. Indeed, in this situation, %we have mentioned in Proposition \ref{stabilityKcv} that 
    both $X$ and $X^*$ have finite cotype, and the property of $T$ to be $\mathcal{R}$-Ritt$_E$ hold for $T^*$ as well. We derive a version of corollary \ref{EquivD} requiring lighter hypothesis.

    \iffalse\begin{corollary}
        Let $X$ be a reflexive and $K$-convex Banach space.
        If $T : X \rightarrow X$ is $\mathcal{R}$-Ritt$_E$ of type $r \in (0,1)$ and if $\alpha,\beta > 0$ are such that $\norm{x}_{T,\alpha} \lesssim \norm{x}$ and $\norm{y}_{T^*,\beta} \lesssim \norm{y}$ for any $(x,y) \in X \times X^*$, then $T$ admits a bounded $\HI(E_s)$ functional calculus for any $s \in (r,1)$.
    \end{corollary}\fi

    \begin{corollary}
        Let $X$ be a reflexive and $K$-convex Banach space. Let $\alpha, \beta >0$ and let $T : X \rightarrow X$ be a Ritt$_E$ operator on $X$. Then the following assertions are equivalent : 
        \begin{enumerate}
            \item $T$ admits a bounded $\HI(E_s)$ functional calculus for some $s \in (0,1)$
            \item $T$ is $\mathcal{R}$-Ritt$_E$ and we have estimates, for any 
            $(x,y) \in X \times X^*$, 
            $$
                \norm{x}_{T,\alpha} \lesssim \norm{x} \quad\hbox{and}\qquad \norm{y}_{T^*,\beta} \lesssim \norm{y}
            $$
        \end{enumerate}
    \end{corollary}

    If $X$ is a UMD space, it is reflexive and $K$-convex. Thus, we can formulate the following statement.

    \iffalse\begin{corollary}
        Let $X$ be a UMD space.
        If $T : X \rightarrow X$ is $\mathcal{R}$-Ritt$_E$ of type $r \in (0,1)$ and if $\alpha,\beta > 0$ are such that $\norm{x}_{T,\alpha} \lesssim \norm{x}$ and $\norm{y}_{T^*,\beta} \lesssim \norm{y}$ for any $(x,y) \in X \times X^*$, then $T$ admits a bounded $\HI(E_s)$ functional calculus for any $s \in (r,1)$.
    \end{corollary}\fi

    \begin{corollary}
        Let $X$ be a UMD space. Let $\alpha, \beta >0$ and $T : X \rightarrow X$ a Ritt$_E$ operator. Then the following assertions are equivalent : 
        \begin{enumerate}
            \item $T$ admits a bounded $\HI(E_s)$ functional calculus for some $s \in (0,1)$
            \item $T$ is $\mathcal{R}$-Ritt$_E$ and we have estimates, for any 
            $(x,y) \in X \times X^*$, 
            $$
                \norm{x}_{T,\alpha} \lesssim \norm{x} \quad\hbox{and}\qquad \norm{y}_{T^*,\beta} \lesssim \norm{y}
            $$
        \end{enumerate}
    \end{corollary}

    When $X$ is a $K$-convex Banach lattice, several assumptions are automatically satisfied and we obtain the following.

    \begin{proposition}
        Let $X$ be a $K$-convex Banach lattice, 
        let $T : X \rightarrow X$ be a Ritt$_E$ operator and let $\alpha, \beta > 0$. The following assertions are equivalent.
        \begin{enumerate}
            \item $T$ admits a bounded $\HI$ functional calculus
            \item $T$ is $\mathcal{R}$-Ritt$_E$ and there exists $C >0$ such that for all $x \in X$ and $y \in X^*$,
                $$
                    \Bignorm{\left(\sum\limits_{k = 0}^\infty k^{2\alpha-1}\abs{T^{k-1}\prod\limits_{j=1}^N(I - \overline{\xi_j}T)^\alpha x}^2\right)^{1/2}} \leq C \norm{x},
                $$
                and
                $$
                    \Bignorm{\left(\sum\limits_{k = 0}^\infty k^{2\beta-1}\abs{(T^*)^{k-1}\prod\limits_{j=1}^N(I - \overline{\xi_j}T^*)^\beta y}^2\right)^{1/2}} \leq C \norm{y}
                $$
        \end{enumerate}
    \end{proposition}

    We give now a last statement in the case where $X = H$ is a Hilbert space.

    \begin{proposition}
        Let $H$ be a Hilbert space, let $T : H \rightarrow H$ be  a Ritt$_E$ operator and 
        let $\alpha, \beta > 0$. The following assertions are equivalent.
        \begin{enumerate}
            \item $T$ admits a bounded $\HI$ functional calculus
            \item There exists $C >0$ such that for all $x,y \in H$,
                $$
                    \left(\sum\limits_{k = 0}^\infty k^{2\alpha-1}\norm{T^{k-1}\prod\limits_{j=1}^N(I - \overline{\xi_j}T)^\alpha x}^2\right)^{1/2} \leq C \norm{x},
                $$
                and
                $$
                    \left(\sum\limits_{k = 0}^\infty k^{2\beta-1}\norm{(T^*)^{k-1}\prod\limits_{j=1}^N(I - \overline{\xi_j}T^*)^\beta y}^2\right)^{1/2} \leq C \norm{y}
                $$
        \end{enumerate}
    \end{proposition}

\newpage

\bibliographystyle{abbrv}	% style utiise pour genere la page des references

\vskip 0.2cm
\end{document}